\documentclass[11pt]{amsart}

\usepackage[letterpaper]{geometry}
\usepackage{fancyhdr}
\allowdisplaybreaks[4]

\usepackage[dvips]{graphicx}
\usepackage{epstopdf}
\usepackage[caption=false]{subfig}
\usepackage{algorithmic}
\ifpdf
  \DeclareGraphicsExtensions{.eps,.pdf,.png,.jpg}
\else
  \DeclareGraphicsExtensions{.eps}
\fi

\usepackage{amsmath,amssymb,amsthm,amsfonts,tikz}
\newtheorem{theorem}{Theorem}[section]
\newtheorem{lemma}{Lemma}[section]

\newtheorem{remark}{Remark}[section]
\newtheorem{example}{Example}%[section]
\newtheorem{definition}{Definition}[section]

\newcommand{\eq}[1]{\begin{align}#1\end{align}}
\newcommand{\eqn}[1]{\begin{align*}#1\end{align*}}
\newcommand{\sss}{\smallskip}

\newcommand{\ii}{\mathbf{i}}
\newcommand{\oo}{\mathcal{O}}
\newcommand{\DD}{\mathcal{D}}
\newcommand{\dd}{\;\mathbf{d}}
\newcommand{\D}{\;\mathrm{d}}
\newcommand{\pe}{\Phi_e}
\newcommand{\po}{\Phi_o}

\newcommand{\xx}{\bm{x}}
\newcommand{\Aga}{\mathcal{A}_\gamma}

\numberwithin{theorem}{section}
\numberwithin{equation}{section}
\numberwithin{figure}{section}

\usepackage{mathrsfs}
\usepackage{bm}
\usepackage{enumerate}

\usepackage{color}
\usepackage{xcolor}
\colorlet{inlinkcolor}{green!50!black}
\colorlet{exlinkcolor}{red!50!black}
\usepackage[colorlinks = true,
  allcolors = inlinkcolor,
  urlcolor = exlinkcolor,
            ]{hyperref}
            
\usepackage{indentfirst}

\title{Dispersion Analysis of CIP-FEM for Helmholtz Equation}
\author{Yu Zhou}
\address[A1]{Department of Mathematics, Nanjing University, Jiangsu,
	210093, P.R. China}
\email{zhouyu524@hotmail.com}
\author{Haijun Wu}
\address[A2]{Department of Mathematics, Nanjing University, Jiangsu,
	210093, P.R. China}
\email{hjw@nju.edu.cn}
\thanks{This work was partially supported by the NSF of China under grants 12171238 and 11525103.}
\begin{document}
\maketitle

%% abstract
\begin{abstract}
	When solving the Helmholtz equation numerically, the accuracy of numerical solution deteriorates as the wave number $k$ increases, known as `pollution effect' which is directly related to the phase difference between the exact and numerical solutions, caused by the numerical dispersion. In this paper, we propose a dispersion analysis for the continuous interior penalty finite element method (CIP-FEM) and derive an explicit formula of the penalty parameter for the $p^{\rm th}$ order CIP-FEM on tensor product (Cartesian) meshes, with which the phase difference  is reduced from $\mathcal{O}\big(k(kh)^{2p}\big)$ to $\mathcal{O}\big(k(kh)^{2p+2}\big)$. Extensive numerical tests show that the pollution error of the CIP-FE solution is also reduced by two orders in $kh$ with the same penalty parameter.
	
\end{abstract}
%%%%%%%%%%%%%%%%%%%%%%%%% Documents %%%%%%%%%%%%%%%%%%%%%%%%%%%%

\section{Introduction}
In many physical applications, such as electromagnetic wave and acoustic scattering problems, are often governed by the Helmholtz equation
\begin{align}
	-\Delta u-k^2u&=f,\quad in\;\;\Omega,\label{eq}\\
	\frac{\partial u}{\partial\mathbf{n}}+\ii ku&=g,\quad on\; \partial\Omega, \label{robin}
\end{align}
where $\Omega\subset\mathbb{R}^d(d=1,2,3)$ is a bounded polygonal/polyhedral domain, $f$ is a given function representing a bounded source of energy, $k>0$ is a constant called the wave number, $\ii=\sqrt{-1}$ denotes the imaginary
unit and $\mathbf{n}$ represents the unit outward normal to $\partial\Omega$. The Robin boundary condition \eqref{robin} is known as the first-order approximation of the following Sommerfeld radiation condition \cite{1979CPApM}. 
$$\lim_{r\rightarrow\infty}r^\frac{d-1}{2}\left(\frac{\partial u}{\partial r}+\ii k u\right)=0,\quad\text{where } r=|x|.$$
Here, it is assumed that the time-harmonic field is $e^{\ii\omega t}$, if the time-harmonic field is instead $e^{-\ii\omega t}$, one should replace $\ii$ with $-\ii$ in the Sommerfeld radiation condition. We remark that the Helmholtz problem \eqref{eq}--\eqref{robin} also arises in applications as a consequence of frequency domain treatment of attenuated scalar waves \cite{1994spacefreq}.

When solving the Helmoholtz equation numerically with classical finite element method, the accuracy of numerical solution deteriorates as the wave number $k$ increases, this effect is what we call `pollution effect' \cite{123d,hp1,hp2}. It arises since the discrete solution fails to propagate waves at the correct speed, resulting in a phase lead/lag in numerical approximation, known as `dispersion'\cite{dispersion}.

Numerical dispersion refers to the difference between exact wave number $k$ and discrete wave number $k_h$, it is widely used in assessing the quality of a numerical scheme. Plenty numerical experiments have shown that the pollution effect is directly related to dispersion, to be more specific, they are of the same convergence order. Though the theoretical proof of association between numerical accuracy and phase difference has been obtained only in limited circumstances, measuring and controlling the numerical dispersion is still of practical significance. 

A method to measure the dispersion on any numerical method is presented in \cite{123d} where the discrete wave number $k_h$ is defined as the solution to a nonlinear equation obtained by some local Fourier analysis. Another definition of $k_h$ is introduced in \cite{mazzieri2012dispersion} where $k_h^2$ an eigenvalue of a Hermitian and positive definite  matrix related to the stiffness matrix of FEM. The explicit form of discrete dispersion relationships for classical finite element mothod (FEM), discontinuous Galerkin finite element discretisation (DGFEM), spectral element method and high-Order N\'{e}d\'{e}lec/edge element approximation are proposed in \cite{Ainsworth2010Explicit,MA2,MA1,MA3}.

Many attempts have been presented in the literature to eliminate/reduce `pollution error' (dispersion error).
\cite{RFbubble,RFbubble2} proposed the `residual free' bubble approach (RF-bubble). \cite{GLS,Harari1991Finite} applied the Galerkin least-squares technology (GLS-FEM)
to the Helmholtz equation, by introducing a local mesh parameter into the variational equation, accurate solutions with relatively coarse meshes was produced. In \cite{2021SoftFEM}, softFEM method was newly coined to reduce the stiffness of the discrete spectral problem. \cite{general,pf} introduced a generalization of the FEM (GFEM), this method covers practically all modifications of the FEM which lead to a sparse system matrix. In one-dimensional case, there exists a pollution-free GFEM solution which is coincide with the best approximation, however, in high dimensional cases, there always exists an equation whose discrete solution contains a pollution term. The paper also derived an effective method for 2D problem (QSFEM), it improves the solution significantly but is also very complicated in general settings.

Our research is based on the continuous interior penalty finite element method (CIP-FEM), which was first proposed by Douglas and Dupont \cite{Douglas} in 1970s to solve elliptic and parabolic problems.  The CIP-FEM uses the same approximation space as that of the FEM but modifies its bilinear form by adding a least squares term penalizing the jump of the gradient of the discrete solution at mesh interfaces, which was also recognized as a stabilization technique \cite{burman2005,burman2007continuous}. Recently, the CIP-FEM has shown great potential in solving wave scattering problems in high frequency \cite{p2cip,cip1,cip2,cip3,liwu2019,wuzou2018}, due to its good stability property and its capability to greatly reduce the pollution errors by tuning the penalty parameters. 

For one-dimensional problems with linear CIP-FEM, it is proved that the relative $H^1$ error of discrete solution $u_h$ could be bounded by best approximation and phase difference \cite{p2cip}, i.e.,
\begin{equation*}\label{cip1}
	\|u-u_h\|_{H^1}\lesssim kh+|k-k_h|,\quad \text{ if }kh\leq 1.
\end{equation*}
In other words, the pollution error  could be bounded only by the phase difference  $|k-k_h|$.
However, the rigorous mathematical proofs of this estimation for high order methods and multi-dimensional cases still remain vague.

In two and three dimensions, the pre-asymptotic error analysis  of CIP-FEM is given in \cite{cip1,cip2,cip3}. 
\begin{equation}\label{error}
	\|u-u_h\|_{H^1}\lesssim (kh)^p+k(kh)^{2p}, \quad \text{if } k(kh)^{2p}\text{ is sufficiently small}
\end{equation}
where $p$ is the order of approximation space.  The first term in \eqref{error} is the local error and the second term is the pollution error which is of the same order as the phase difference. By selecting appropriate penalty parameter the `pollution effect' could be eliminated in 1D and  largely reduced in 2D \cite{cip1,cip3}. However, searching for appropriate penalty parameters 
involves massive calculations, especially for multidimentional cases and high order finite element schemes.

The dispersion analysis for classical FEM ($hp$-version) has been done by Mark Ainsworth \cite{MA1}, where the following explicit characterization of the phase difference for elements of arbitrary order is derived:
\begin{align*}\label{ma1}
	|k-k_h|=\frac 12\left[\frac{p!}{(2p)!}\right]^2\frac{k^{2p+1}h^{2p}}{(2p+1)}+\oo(k^{2p+3}h^{2p+2}).
\end{align*}
In this research, the dispersion relation is first obtained by decoupling the nodal and interior degrees of freedom through Gaussian elimination or static condensation \cite{dispersion,hp2} and then expressed explicitly in terms of Pad\'e approximants. However, this approach fails in CIP-FEM since the penalty terms cause the nodal and interior degrees of freedom can not be decoupled.

The purpose of this paper is to conduct the dispersion analysis for the CIP-FEM on tensor product (Cartesian) meshes with the interior penalty term involving only the jumps of $p^{\rm th}$ normal derivative. We use the method developed in \cite{123d} to measure the dispersion and use the same idea of static condensation used in \cite{MA1} to do some simplification. While the result dispersion relation for the CIP-FEM is still more complicated than that of FEM \cite{MA1}, due to the difficulty caused by non-decoupling. 
Some subtle and tedious manipulation yields the following  characterization of the phase difference for the $p^{\rm th}$ order CIP-FEM in $\mathbb{R}^d (d=1,2,3)$.
\begin{align*}
	|k-k_h|=\frac 12\bigg(\frac{1}{(2p+1)}\left[\frac{p!}{(2p)!}\right]^2+\gamma\bigg)k^{2p+1}h^{2p}+\oo(k^{2p+3}h^{2p+2}),
\end{align*}
where $\gamma$ is the penalty parameter. Therefore by taking 
\begin{align*}
	\gamma=\gamma_0:=-\frac{1}{(2p+1)}\left[\frac{p!}{(2p)!}\right]^2,
\end{align*}
the phase difference may be reduced from $\oo\big(k(kh)^{2p}\big)$ to $\oo\big(k(kh)^{2p+2}\big)$. 
Note that adding penalty terms on jumps of derivatives lower than $p$ (for $p\ge 2$) may reduce further the phase error \cite{cip3}, while explicit formulas for the penalty parameters are not easy to find for general $p$. We will investigate this in a future work.

The rest of the paper is organized as follows. In \S 2, we address the model problem and the definition of discrete wave number. \S 3 is devoted to the dispersion analysis for CIP-FEM in one-dimensional case. The dispersion analysis is  then extended to two- and three-dimensional cases in \S 4. Some numerical results are given in \S 5 to verify the theoretical findings.
Throughout this paper, let $C$ denote a generic positive constant which is independent of $k, h, f, g$, which may have different values in different occasions.

\section{CIP-FEM and discrete wave number}
In this section, we introduce the formulation of the CIP-FEM and the definition of the discrete wave number.

\subsection{CIP-FEM}
We start from the Helmholtz equation in $\mathbb{R}^d$
\begin{equation}\label{equation}
	-\Delta u-k^2u=f\quad \text{ in }\mathbb{R}^d,
\end{equation}
where $k=\frac{2\pi}{\lambda}$ is the wave number describing how many oscillations a wave completes per unit of space, $f$ is a source function. Since the goal of this analysis is to derive the dispersion 
relations, we make several assumptions \cite{mazzieri2012dispersion}. We assume that the medium occupies an unbounded region which is isotropic (i.e., looking the same in all directions), homogeneous (i.e., the same at each place) and source free (i.e., $f\equiv 0$). Moreover, it follows logically to assume $u\rightarrow 0$ for all $|\bm{x}|\rightarrow\infty$. Under these assumptions, by taking a dot product of \eqref{equation} with a sufficient smooth test function $v$ of compact support, integrating over $\mathbb{R}^d$ on both sides and applying Green's formula, we come to the variational form
\begin{equation*}
	\mathcal{A}(u,v):=(\nabla u,\nabla v)- k^2( u,v)=0,
\end{equation*}
where $(\cdot,\cdot)$ denotes the $L^2$ inner product on $L^2(\mathbb{R}^d)$. 

To obtain the CIP-FEM scheme of \eqref{equation}, we introduce the following notations \cite{cip1,cip2,cip3}.

Suppose $\mathbb{R}^d$ is decomposed into non-overlapping d-cube (a $d$-dimensional cube degenerates to a line segment in 1D and a square in 2D) elements $\{K_i\}_{i\in I}$ with equal size $h$, denoted by $\mathcal M_h$. Let
\begin{align*}
	\mathcal{E}_h&:=\text{the set of all $(d-1)$-faces of $d-$cube  elements in}  \mathcal{M}_h,\\
	\mathbb{Q}_p&:=\text{the set of all polynomials with degree $\le$ $p$ in each variable},\\
	V_h&:=\left\{v_h\in H^1_{loc}(\mathbb{R}^d):\; v_h|_K\in\mathbb{Q}_p,\;\forall K\in \mathcal{M}_h\right\},\\
	\mathcal{N}_h&:=\text{the set of all global nodes of the finite element space  $V_h$},\\
	\Phi\;&:=\Big\{\phi_{\bm{x}}\in \mathbb{Q}_p:\; \phi_{\bm{x}}(\bm{x})=1,\; \phi_{\bm{x}}(\bm{x'})=0,\; \forall  \bm{x}\neq\bm{x}'\in \mathcal{N}_h \Big\}.
\end{align*}
Set the penalty term as
$$	J(u,v):=\sum_{e\in\mathcal{E}_h}\gamma h^{2p-1}\int_e\left[\frac{\partial^p u}{\partial \mathbf{n}^p}\right] \left[\frac{\partial^p v}{\partial \mathbf{n}^p}\right],$$
where $\gamma$ is the penalty parameter, the jump $[v]$ of $v$ on an interior face $e=\partial K_1\cap\partial K_2\in\mathcal{E}_h$ is defined by
$$[v]\vert_e:=v|_{K_1}\cdot\mathbf{n}_{K1}+v|_{K_2}\cdot\mathbf{n}_{K2},$$
$\mathbf{n}_{K_i}$ is the unit outward normal  towards $\partial K_i$.

Note that $J(u, v) = 0$ if $u\in H^{p+1}(\mathbb{R}^d)$ is a solution to \eqref{equation}, thus there still holds
\begin{equation}\label{vp1}
	\Aga(u,v):=\mathcal{A}(u,v)+J(u,v)=(\nabla u,\nabla v)- k^2( u,v)+J(u,v)=0,
\end{equation}

By analogy with the continuous problem, the CIP-FE solution  $u_h \in V_h$ satisfies (see e.g. \cite{cip3})
\begin{equation}\label{cip_fem}
	\Aga(u_h,v_h):=(\nabla u_h,\nabla v_h)- k^2(u_h,v_h)+J(u_h,v_h)=0,\quad\forall v_h\in V_h.
\end{equation}

\begin{remark}
	{\rm (a)}  The
	CIP-FEM was first proposed by Douglas and Dupont \cite{Douglas} for elliptic and parabolic problems in the 1970s and then successfully applied to con-vection-dominated problems as a stabilization technique \cite{burman2005,burman2007continuous}.
	
	{\rm (b)} By choosing appropriate penalty parameter, the pollution error could be
	eliminated in one dimension  and largely reduced in two or more dimensions \cite{cip2,cip3}. Moreover, the scheme is absolutely stable if the penalty parameter is a complex number with positive imaginary part \cite{cip1}. While in the dispersion analysis of this paper, for simplicity, we assume that $\gamma$ is real.
	If $\gamma=0$, the CIP-FEM scheme becomes the classical FEM discretization. 
	
	{\rm (c)} Compared to the discontinuous Galerkin
	methods \cite{feng2011hp,feng2009discontinuous} and  hybridizable discontinuous Galerkin method \cite{chen2013hybridizable}, the CIP-FEM involves fewer degrees of freedom (DOF), and thus reduce the computational cost.
	
	{\rm (d)} Compared to the $p^{\rm th}$ order CIP-FEM proposed in \cite{cip3}, we take only the penalty term on the jumps of highest order normal derivative and omit the penalty terms on jumps of lower order normal derivatives. Although more penalty terms can help to reduce further the phase error and the pollution effect, explicit formulas for the penalty parameters are not easy to find. We leave this to the future investigation.
\end{remark}

\subsection{Discrete wave number}
It is clear that the homogeneous Helmholtz equation
\begin{equation}\label{homoequation}
	-\Delta u-k^2u=0,
\end{equation}
admits a plane wave solution in the form of
$$u(\bm{x})=Ae^{\ii \bm{k}\cdot\bm{x}},$$
if $\bm{k}=(k_1,\cdots,k_d)$ satisfies the following dispersion relationship
$$k=|\bm{k}|,$$
furthermore, the exact solution $u(\bm{x})$ is a Bloch wave \cite{bloch} satisfying
\eq{\label{bloch} u(\bm{x}+\bm{m}h)=e^{\ii \bm{m}\cdot\bm{k}h}u(\bm{x}),\quad \forall \bm{m}\in \mathbb{Z}^d.}

In order to define the discrete wave number and carry out the dispersion analysis of CIP-FEM, we first introduce the definition of the generating set of global nodes of a finite element space on a tensor product mesh as follows.
\begin{definition}[Generating Set]\label{def_generate}
%	Let $\mathcal{M}_h$ be a tensor product mesh of $\mathbb{R}^d$ consisting of cubes of size $h$,  $V_h$ be the $p^{\rm th}$ order finite element space on $\mathcal{M}_h$, and $\mathcal{N}_h$ be the set of all global nodes of $V_h$. 
Let $\mathcal{N}_h$ defined as above,
	we say that two nodes $\bm{x}, \bm{y}\in\mathcal{N}_h$ are equivalent and denoted by $\bm{x}\sim\bm{y}$ if $\bm{x}-\bm{y}\in h \mathbb{Z}^d$.  We call a subset $\mathcal{X}_g\subset\mathcal{N}_h$ a generating set of $\mathcal{N}_h$ if {\rm (i)} any node in $\mathcal{N}_h$ is equivalent to a certain node in $\mathcal{X}_g$; {\rm (ii)} any two nodes in $\mathcal{X}_g$ are not equivalent.
\end{definition}
\begin{remark}
	{\rm (a)} It is clear that   $\mathcal{X}_g$ contains $p^d$ nodes.
	
	{\rm (b)} The generating set is not unique. For example, if $\mathcal{X}_g$ is a generating set of $\mathcal{N}_h$, the set obtained by replacing any node in $\mathcal{X}_g$ by one of its equivalent nodes is still a generating set of $\mathcal{N}_h$ (see Figure~\ref{2dgs}).
	
	{\rm (c)} The definition of generating set may be extended to other translation-invariant meshes (e.g. equilateral triangulations in 2D \cite{2dperiod} and tetrahedral meshes in 3D\cite{3dlinear}).
\end{remark}
\begin{figure}
	\centering
	\includegraphics[scale=0.75,trim={0cm 1cm 0cm 0}]{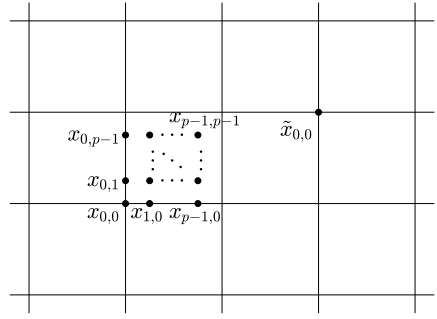}
	\caption{Illustration of generating sets on the 2D tensor product mesh. Both   $\mathcal{X}_g=\{\bm{x}_{0,0},\bm{x}_{0,1},\cdots,\bm{x}_{p-1,p-1}\}$ and $\tilde{\mathcal{X}}_g=\{\tilde{\bm{x}}_{0,0},\bm{x}_{0,1},\cdots,\bm{x}_{p-1,p-1}\}$ are generating sets of $\mathcal{N}_h$.}
	\label{2dgs}
\end{figure}

We apply the method developed in \cite{123d} to measure the dispersion. Since the mesh is translation-invariant, we consider only the equations associated to the generating set. Denote by $n_g:=\#\mathcal{X}_g=p^d$. Write
\eq{\label{NX}\mathcal{N}_h=\{\bm{x}_1,\bm{x}_2,\cdots\}\quad\text{and}\quad \mathcal{X}_g=\{\bm{x}_g^1,\bm{x}_g^2,\cdots,\bm{x}_g^{n_g}\}.}
clearly, the CIP-FE solution may be expressed  as 
\eqn{u_h=\sum_{\bm{x}_j\in\mathcal{N}_h}U_j\phi_{\bm{x}_j}, \quad U_j=u_h(\xx_j).}
For any $\bm{x}_g^s\in \mathcal{X}_g$ and $\bm{x}_j\in\mathcal{N}_h$, denote that 
\eqn{
	D_j^s&=\Aga(\phi_{\xx_j},\phi_{\xx_g^s})=(\nabla \phi_{\xx_j},\nabla \phi_{\xx_g^s})- k^2(\phi_{\xx_j},\phi_{\xx_g^s})+J(\phi_{\xx_j},\phi_{\xx_g^s}).} 
by taking $v_h=\phi_{\bm{x}_g^s}$ in \eqref{cip_fem}, we obtain the CIP-FE equation associated to  $\bm{x}_g^s$, which  can be written as follows:
\begin{align}\label{eqs}
	\sum_{\bm{x}_j\in \Lambda^s}D_j^sU_j=\sum_{q=1}^{n_g}\sum_{\bm{x}_j\in \Lambda_q^s}D_{j}^sU_j=0,\quad s=1, \cdots, n_g,
\end{align}
where
\begin{align*}
	\Lambda^s=\left\{\bm{x}_j\in\mathcal{N}_h:\;D_j^s\neq 0\right\}\quad\text{and}\quad \Lambda_q^s=\{\bm{x}_j\in\Lambda^s:\; \bm{x}_j\sim \bm{x}_g^q\}.
\end{align*}

By analogy with the continuous solution (see \eqref{bloch}), the invariance of grid prompts us to seek solutions satisfying the Bloch wave condition
\begin{align}\label{bloch_h}
	u_h(\bm{x}+\bm{m}h)=e^{\ii \bm{m}\cdot\bm{k}_h h}u_h(\bm{x}).
\end{align}
under this assumption, $\{U_j\}_{\bm{x}_j\in \Lambda^s}$ could be reperesented by  $$\mathbf{U}_g=\big(u_h(\bm{x
}_g^1),u_h(\bm{x
}_g^2),\cdots,u_h(\bm{x
}_g^{n_g})\big),$$ hence \eqref{eqs} leads to a  system of $n_g$ equations. 
\begin{align}\label{equarray}
	\mathcal{D}\mathbf{U}_g=0,
\end{align}
which admits nontrivial solution only if 
\begin{align}\label{kh}
	\mathbf{Det}(\mathcal{D})=0,
\end{align}

The explicit expression of $\mathcal{D}$ will be given later in the next two sections. 	Since $D_{j}^s$ are functions of $k$ and $h$, the forementioned equation derives a relationship between $k$ and $\bm{k}_h$. Let $k_h=|\bm{k}_h|$, by using spherical co-ordinates in $\mathbb{R}^d$, 
\begin{align*}
	\left\{\begin{array}{ll}
		k_{h1}&=k_h \cos \theta_1\\
		k_{h2}&=k_h \sin \theta_1\cos \theta_2\\
		&\cdots\\
		k_{hd}&=k_h \sin \theta_1\sin \theta_2\cdots\sin \theta_{d-1}
	\end{array}\right.,
\end{align*}
the difference $|k-k_h|$ is a function of $\theta_1,\cdots,\theta_{d-1}$ which measures the dispersion in various directions. Note that in multi-dimensional problems, we define the phase difference as the upper bound of $|k-k_h|$ with respect to $\theta_1,\cdots,\theta_{d-1}$.

We remark that the non-uniqueness of the generating set does not effect the definition of the discrete wave number. For example, for the two generating sets in Figure~\ref{2dgs}, the equation at $\tilde{\bm{x}}_{0,0}$ in \eqref{equarray}  corresponding to $\tilde{\mathcal{X}}_g$  could be obtained by multiplying a non-zero factor by the equation at $\bm{x}_{0,0}$ in \eqref{equarray} corresponding to $\mathcal{X}_g$.

\section{Dispersion analysis in one dimension}
\begin{figure}
	\centering
	\includegraphics[scale=0.75,trim={0cm 3cm 0cm 0}]{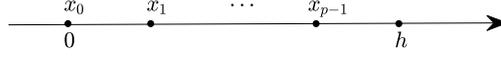}
	\caption{Generating set on $\mathbb{R}$.}
	\label{1dgs}
\end{figure}

In this section we carry out dispersion analysis for the CIP-FEM for the one-dimensional problem. For simplicity, we suppose the mesh $\mathcal{M}_h=\{[(j-1)h, jh]:\; j\in\mathbb{Z}\}$. Then the set of global nodes of the $p^{\rm th}$ order FE space $V_h$ is $\mathcal{N}_h=\{x_m:=mh/p:\;m\in\mathbb{Z}\}$. According to Definition~\ref{def_generate}, a generating set of $\mathcal{N}_h$ is $\mathcal{X}_g=\{x_0,\cdots,x_{p-1}\}$ as shown in Figure~\ref{1dgs}

The coefficient matrix $\mathcal{D}$ in \eqref{equarray} is a $p\times p$ matrix whose explicit form is given by the following lemma. 
\begin{lemma}\label{1deq} Suppose $\gamma\in\mathbb{R}$. 
	When solving the one-dimensional Helmholtz equation with $p^{\rm th}$ order CIP-FEM, the coefficient matrix $\mathcal{D}=\big(\DD_{i,j}^{t,t_h}\big)_{p\times p}$ associated to the generating set $\mathcal{X}_g=\{x_0,\cdots,x_{p-1}\}$ takes the following form: for $1\le i,j\le p-1$,
	\begin{equation*}
		\begin{array}{lll}
			\DD_{1,1}^{t,t_h} &=\left\{
			\begin{array}{ll}
				\begin{split}
					2B_t(\lambda_0,\lambda_0)&+2\cos t_hB_t(\lambda_p,\lambda_0)\\
					&+2p^{2p}\big(1-\cos(2t_h)\big)\gamma,
				\end{split}		 &\hspace{17mm} p  \text{ even},\\
				\begin{split}
					2B_t(\lambda_0,\lambda_0)&+2\cos t_hB_t(\lambda_p,\lambda_0)\\
					&+2p^{2p}\big(3-4\cos t_h+\cos(2t_h)\big)\gamma,
				\end{split}&\hspace{17mm} p \text{ odd},
			\end{array} \right.\\
			\DD_{1,j+1}^{t,t_h}&=\left\{
			\begin{array}{lll}
				\begin{split}
					B_t(\lambda_j,\lambda_0)&+e^{-\ii t_h}B_t(\lambda_{j},\lambda_p)\\
					&+{(-1)}^j\tbinom pjp^{2p}
					(1+e^{-\ii t_h}-e^{\ii t_h}-e^{-2\ii t_h})\gamma,
				\end{split}&p\text{ even},\\
				\begin{split}
					B_t(\lambda_j,\lambda_0)&+e^{-\ii t_h}B_t(\lambda_{j},\lambda_p)\\
					&+{(-1)}^j\tbinom pjp^{2p}
					(3-3e^{-\ii t_h}-e^{\ii t_h}+e^{-2\ii t_h})\gamma,
				\end{split}&p \text{ odd},
			\end{array} \right.\\
			\DD_{i+1,1}^{t,t_h}&= \mathbf{conj}(\DD_{1,i+1}^{t,t_h}),\\
			\DD_{i+1,j+1}^{t,t_h}&=B_t(\lambda_i,\lambda_j)+2{(-1)}^{i+j}\tbinom pi\tbinom pj p^{2p}(1-\cos t_h)\gamma,
		\end{array}
	\end{equation*}
	where $t=kh,t_h=k_h h$,  $B_t(u,v)=\int_0^1(u'v'-t^2uv)\D x$, and $\{\lambda_i\}_{0\leq i\leq p}$ is the nodal basis of the $p^{\rm th}$ order Lagrange finite element on $[0,1]$, i.e., $\lambda_i\in\mathbb{Q}_p$ satisfies $\lambda_i(\frac{j}{p})=\delta_{ij}(0\leq i,j\leq p)$.
\end{lemma}
\begin{proof}
	For simplicity, denote by $\phi_i=\phi_{x_i}$ the nodal basis function at $x_i$. It is clear that $\phi_i(x)=\lambda_i(\frac{x}{h})$ for $x\in [0,h]$ and $0\le i\le p$. By change of variable, we have
	\begin{align}\label{em1}
		\int_0^h(\phi_i'\phi_j'-k^2\phi_i\phi_j)\D x
		=\frac{1}{h}B_t(\lambda_i,\lambda_j), \quad 0\le i,j \le p.
	\end{align}
	
	On the other hand,  from  the Lagrange interpolation formula,  $$\lambda_i(x)={\prod_{{j=0}\atop{j\neq i}}^p\Big(x-\frac jp\Big)}\big/{\prod_{{j=0}\atop{j\neq i}}^p\Big(\frac ip-\frac jp\Big)},$$ and hence
	\begin{align}\label{em2}
		\phi_i^{(p)}=h^{-p}\lambda_i^{(p)}=h^{-p}{(-1)}^{p-i}\tbinom pip^p,\quad x\in [0,h],\; 0\le i\le p.
	\end{align}
	
	Next we consider the CIP-FE equation associated to $x_0=0$. From the FE scheme \eqref{vp1}, Bloch wave condition \eqref{bloch_h}, the identites \eqref{em1}--\eqref{em2}, and the fact that $\phi_{i+mp}(x)=\phi_i(x-mh),\forall m\in\mathbb{Z}$,  we derive that
	\begin{align*}
		0=&\Aga(\phi_0,\phi_0)U_0+\Aga(\phi_p,\phi_0)(U_p+U_{-p})+\Aga(\phi_{2p},\phi_0)(U_{2p}+U_{-2p})\\
		&+\sum_{j=1}^{p-1}\Big(\Aga(\phi_j,\phi_0)U_j+\Aga(\phi_{j-p},\phi_0)U_{j-p}\\
		&\qquad+\Aga(\phi_{j+p},\phi_0)U_{j+p}+\Aga(\phi_{j-2p},\phi_0)U_{j-2p}\Big)\\
		=&\Big(\mathcal{A}(\phi_0,\phi_0)+J(\phi_0,\phi_0)+2\cos t_h\big(\mathcal{A}(\phi_p,\phi_0)+J(\phi_p,\phi_0)\big)\\
		&\qquad+2\cos(2t_h)J(\phi_{2p},\phi_0)\Big)U_0\\
		&+\sum_{j=1}^{p-1}\Big(\mathcal{A}(\phi_j,\phi_0)+J(\phi_j,\phi_0)+e^{-\ii t_h}\big(\mathcal{A}(\phi_{j-p},\phi_0)+J(\phi_{j-p},\phi_0)\big)\\
		&\qquad+e^{\ii t_h}J(\phi_{j+p},\phi_0)+e^{-2\ii t_h}J(\phi_{j-2p},\phi_0)\Big)U_j\\
		%%%%%%%%%%%%%%%%%%%%%%%%%%%
		=&\frac{1}{h}\bigg(2B_t(\lambda_0,\lambda_0)+\Big(2+\big(1-(-1)^p\big)^2\Big)p^{2p}\gamma\\
		&\qquad+2\cos t_h\Big(B_t(\lambda_p,\lambda_0)-\big(1-(-1)^p\big)^2p^{2p}\gamma\Big)-2\cos(2t_h)(-1)^p p^{2p}\gamma \bigg)U_0\\
		&+\frac{1}{h}\sum_{j=1}^{p-1}\bigg(B_t(\lambda_j,\lambda_0)+(-1)^{p-j}\tbinom{p}{j}\big(2(-1)^p-1\big)p^{2p}\gamma\\
		&\qquad+e^{-\ii t_h}\Big(B_t(\lambda_j,\lambda_p)+(-1)^{p-j}\tbinom{p}{j}\big(2-(-1)^p\big)p^{2p}\gamma\Big)\\
		&\qquad-e^{\ii t_h}(-1)^{2p-j}\tbinom{p}{j}p^{2p}\gamma -e^{-2\ii t_h}(-1)^{p-j}\tbinom{p}{j}p^{2p}\gamma \bigg)U_{j}\\
		%%%%%%%%%%%%%%%%%%%%%%%%%
		=&\frac{1}{h}\bigg(2B_t(\lambda_0,\lambda_0)+2\cos t_hB_t(\lambda_p,\lambda_0)\\
		&\quad+\Big(2+\big(1-(-1)^p\big)^2-2\big(1-(-1)^p\big)^2\cos t_h-2(-1)^p\cos(2t_h)\Big)p^{2p}\gamma\bigg)U_0\\
		&+\frac{1}{h}\sum_{j=1}^{p-1}\Big\{B_t(\lambda_j,\lambda_0)+e^{-\ii t_h}B_t(\lambda_{j},\lambda_p)\\
		&\quad+\Big(\big(2-(-1)^p\big)+\big(2(-1)^p-1\big)e^{-\ii t_h}-e^{\ii t_h}-(-1)^p e^{-2\ii t_h}\Big)(-1)^{j}\tbinom pj p^{2p}\gamma\Big\} U_{j},
	\end{align*}
	which implies that the first two formulas hold. To prove the last two formulas, we consider the equations associated to $x_i$ for $1\le i\le p-1$. Similar as above, we have
	\begin{align*}
		0=&\Aga(\phi_0,\phi_i)U_0+\Aga(\phi_p,\phi_i)U_p+J(\phi_{-p},\phi_i)U_{-p}+J(\phi_{2p},\phi_i)U_{2p}\\
		&+\sum_{j=1}^{p-1}\Big(\Aga(\phi_j,\phi_i)U_j+J(\phi_{j-p},\phi_i)U_{j-p}+J(\phi_{j+p},\phi_i)U_{j+p}\Big)\\
		=&\frac{1}{h}\bigg(B_t(\lambda_0,\lambda_i)+e^{\ii t_h}B_t(\lambda_p,\lambda_{i})\\
		&\quad+\Big(\big(2-(-1)^p\big)+\big(2(-1)^p-1\big)e^{\ii t_h}-e^{-\ii t_h}-(-1)^p e^{2\ii t_h}\Big)(-1)^i\tbinom p i p^{2p}\gamma\bigg) U_0\\
		&+\frac{1}{h}\sum_{j=1}^{p-1}\bigg(B_t(\lambda_j,\lambda_i)+(-1)^{i+j}\tbinom p i\tbinom p j(2-2\cos t_h) p^{2p}\gamma\bigg)U_{j},
	\end{align*}
	which implies the last two formulas.
	This completes the proof of Lemma~\ref{1deq}.
\end{proof}
Next we turn to analyze  $\mathbf{Det}(\DD^{t,t_h})$ but it is hard to do so by using the explicit form given in the above lemma. We have to do some simplifications. Notice that, for FEM (i.e. $\gamma=0$), since the nodal degrees of freedom at $x_{mp} (m\in\mathbb{Z})$ and the interior ones  at $x_{mp+j}(1\le j\le p-1)$ can be decoupled, the $(1,i)^{\rm th},(i,1)^{\rm th}(2\leq i\leq p)$ entries in $\DD^{t,t_h}$ (with $\gamma=0$) can be eliminated by Gaussian elimination or static condensation \cite{MA1,dispersion,hp2}. Although such a procedure for FEM can not eliminate those entries for CIP-FEM (with $\gamma\neq 0$), it does transform the matrix to another simpler and more operable form.  This procedure is equivalent to modified the basis functions at mesh points $x_0$ and $x_p$ as follows (cf. \cite{MA1}). Let
\begin{align*}
	\xi_0&:=\lambda_0+\sum_{i=1}^{p-1}c_i\lambda_i,\\
	\xi_1&:=\lambda_p+\sum_{i=1}^{p-1}d_i\lambda_i,
\end{align*}
$c_i,d_i$ are functions of $t$, such that
\begin{align}\label{defcidi}
	\left\{\begin{array}{ll}
		B_t(\xi_0,\lambda_i)&=0\\
		B_t(\xi_1,\lambda_i)&=0
	\end{array} \right.,\quad 1\leq i\leq p-1.
\end{align}
The existence and uniqueness of $c_i$ and $d_i$ hold if $t$ is not a discrete eigenvalue, in particular,  if $t$ is sufficiently small as a consequence of Lemma~\ref{positive} below.

\begin{lemma}\label{positive}
	Let $D_0:=\big(B_0(\lambda_i,\lambda_j)\big)_{1\leq i,j\leq p-1}=\big(\int_0^1\lambda_i'\lambda_j'\D x\big)_{1\leq i,j\leq p-1}$. Then 
	$$\alpha_0:=\mathbf{Det}(D_0)>0.$$ 
\end{lemma}
\begin{proof}
	Although the proof is trivial, the result is of major importance. For any $\bm{v}=(v_1,\cdots,v_{p-1})^T$ $\in\mathbb{R}^{p-1}$, let
	$v=\sum_{i=1}^{p-1}v_i\lambda_i$, according to Poincar\'{e} inequality on $H_0^1(0,1)$,
	$$\bm{v}^TD_0\bm{v}=\lVert v'\rVert_{L_2(0,1)}^2\ge \pi^2\lVert v\rVert_{L_2(0,1)}^2\gtrsim|\bm{v}|^2,$$
	thus $D_0$ is positive definite. which completes the proof of Lemma~\ref{positive}.
\end{proof}

In order to make the structure of the article clear, we put the proofs of the following three lemmas in Appendices~\ref{A1} to \ref{A3}, repectively.

By using  \cite[Theorem 3.1, Theorem 4.1 and Theorem 4.2]{MA1}, we may prove the following two lemmas which give explicit forms of the basis functions $\xi_0, \xi_1$, and the coefficients $c_i$ and $d_i$. The proofs are given in Appendices~\ref{A1} and \ref{A2}, respectively.
\begin{lemma}\label{pepo}
	The explicit form of $\xi_0$ and $\xi_1$ reads:
	$$\xi_0=\frac{1}{2}(\pe-\po),\quad\xi_1=\frac{1}{2}(\pe+\po),$$
	where
	\begin{align}\label{exppepo}
		\begin{array}{l}
			\pe(x):=\sum_{j=1}^{N_e+1}\Big\{(-1)^jt^{-2j}\frac{2(2N_e+1)!}{(2N_e+2-2j)!}\sum_{m=0}^{2N_e+2-2j}\tbinom{2N_e+2j}{m-1+2j}\tbinom{2N_e+2-2j}{m}\\
			\qquad\qquad x^{2N_e+2-2j-m}(x-1)^m\Big\}\bigg/\sum_{j=1}^{N_e+1}\Big\{(-1)^jt^{-2j}{\frac{2(2N_e+2j)!}{(2N_e+2-2j)!(2j-1)!}}\Big\},\\
			%%%%%%%%%%%%%%%%%%%%%%%%%%%%%%
			\po(x):=\sum_{j=1}^{N_o}\Big\{(-1)^jt^{-2j}\frac{2(2N_o)!}{(2N_o+1-2j)!}\sum_{m=0}^{2N_o+1-2j}\tbinom{2N_o-1+2j}{m-1+2j}\tbinom{2N_o+1-2j}{m}\\
			\qquad\qquad x^{2N_o+1-2j-m}(x-1)^m\Big\}\bigg/\sum_{j=1}^{N_o}\Big\{(-1)^jt^{-2j}{\frac{2(2N_o-1+2j)!}{(2N_o+1-2j)!(2j-1)!}}\Big\},\\
			N_e :=  \big\lfloor \frac{p}2 \big\rfloor,N_o := \big\lfloor \frac{p+1}2\big\rfloor.
		\end{array}
	\end{align}
	we also have the following estimates:
	\begin{align}\label{btpepo}
		\begin{array}{l}
			B_t(\pe,\pe)=-2t\tan{\frac{t}{2} }+{\Big[\frac{(2N_e+1)!}{(4N_e+2)!}\Big]}^2 \frac{t^{4N_e+4}}{4N_e+3}+\mathcal O(t^{4N_e+6}),\hspace{4mm} t \neq m\pi,m \in \mathbb{Z},\vspace{1mm}\vspace{2mm}\\
			B_t(\po,\po)=\hspace{3mm}2t\cot{\frac{t}{2} }+4{\Big[\frac{(2N_o)!}{(4N_o)!}\Big]}^2 \frac{t^{4N_o}}{4N_o+1}+\mathcal O(t^{4N_o+2}),\hspace{6mm} t \neq m\pi,m \in \mathbb{Z},\vspace{2mm}\\
		\end{array}
	\end{align}
\end{lemma}

\begin{lemma}\label{cidi}
	\begin{align*}
		\begin{array}{cll}
			c_i&=&\xi_0(\frac{i}{p})=\frac{1}{2}\big(\pe(\frac i p)-\po(\frac i p)\big),\quad 1\leq i\leq p-1,\\
			d_i&=&\xi_1(\frac{i}{p})=\frac{1}{2}\big(\pe(\frac i p)+\po(\frac i p)\big),\quad 1\leq i\leq p-1,\vspace{1mm}\\
			A_1:&=&\sum_{i=1}^{p-1}{(-1)}^i c_i\tbinom pi\vspace{1mm}\\
			&=&\frac{{(-1)}^p}{\sum_{j=1}^{\lfloor\frac p2\rfloor+1}{(-1)}^jt^{-2j}\frac{4(p+2j)!}{(p+2-2j)!(2j-1)!}}\sum_{j=1}^{\lfloor\frac p2\rfloor+1}\Big\{{(-1)}^jt^{-2j}
			\frac{2(p+1)!}{(p+2-2j)!p^{p+2-2j}}\vspace{1mm}\\
			&&\sum_{i=1}^{p-1}{(-1)}^i \tbinom pi\sum_{m=0}^{p+2-2j} \tbinom{p+2j}{m-1+2j} \tbinom{p+2-2j}{m}
			i^{p+2-2j-m}(i-p)^m\Big\},\\
			A_2:&=&\sum_{i=1}^{p-1}{(-1)}^i d_i\tbinom pi=(-1)^p A_1.
		\end{array}
	\end{align*}
\end{lemma}

The following lemma is used to simplify $A_1$ and $A_2$ in the above lemma, which  can be derived in virtue of the combination formulas stated in \cite{table}. The proof is given in  Appendix~\ref{A3}.
\begin{lemma}\label{combination}
	\begin{align*}
		\mathcal N:=&\sum_{i=1}^{p-1}{(-1)}^i\tbinom{p}{i}\sum_{m=0}^{p+2-2j}\tbinom{p+2j}{m-1+2j}\tbinom{p+2-2j}{m}
		i^{p+2-2j-m}{(i-p)}^m,\\
		=&\left\{
		\begin{array}{lc}
			2{(-1)}^p \left(\frac{(2p+1)!}{(p+1)!}-(p+2)p^p\right),&j=1,\vspace{3mm}\\			
			2{(-1)}^{p+1}\tbinom{p+2j}{2j-1}p^{p+2-2j},&2\leq j\leq \lfloor{\frac p 2}\rfloor+1.
		\end{array} \right.
	\end{align*}
\end{lemma}
With the help of the preceding three lemmas, we are now in the position to construct the transformation matrix $Q$ to simplify the matrix $\mathcal D^{t,t_h}$. Given $\beta\in\mathbb{R}$,	let
\begin{align}\label{Q}
	Q^\beta :=&\left(
	\begin{array}{cccc}
		1 & c_1+e^{-\ii \beta}d_1 &\cdots& c_{p-1}+e^{-\ii \beta}d_{p-1}\\
		0&1&\cdots&0\\
		\vdots&\vdots &\ddots&\vdots\\
		0&0&\cdots&1
	\end{array}
	\right).
\end{align}

The following lemma implies that the congruent transform of the matrix $\mathcal D^{t,t_h}$ by $(Q^{t_h})^\mathbf{H}$ (the conjugate transpose of $Q^{t_h}$) changes the $(1,i)^{\rm th},(i,1)^{\rm th}(2\leq i\leq p)$ entries in $\DD^{t,t_h}$ to higher order terms in $t$ and $t_h$.
\begin{lemma}	\label{tildeD} The matrix
	$$\tilde{\mathcal D}^{t,t_h}:=Q^{t_h}{\mathcal D}^{t,t_h} (Q^{t_h})^\mathbf{H},$$	
	satisfies the following estimates: for $1\leq i,j\leq p-1$,
	\begin{subequations}
		\begin{align}
			&\tilde{\mathcal D}_{1,1}^{t, t}
			=\bigg(\frac{1}{2p+1}\Big(\frac{p!}{(2p)!}\Big)^2+\gamma\bigg)t^{2p+2}+\oo(t^{2p+4}),\label{dija}\\
			&\tilde{\mathcal D}_{1,j+1}^{t,t}=\oo(t^{p+2}),\qquad\qquad\qquad\tilde{\mathcal D}_{i+1,1}^{t,t}=\oo(t^{p+2}),\label{dijc}\\
			&\tilde{\mathcal D}_{i+1,j+1}^{t,t}
			=\int_0^1\lambda_i'\lambda_j'+\oo(t^2),\label{dijd}\\
			&\frac{\partial \tilde{D}^{t,t_h}_{1,1}}{\partial t_h}\bigg\vert_{t_h=t}=2t+\oo(t^{2p+1}),\\
			&\frac{\partial \tilde{D}^{t,t_h}_{1,j+1}}{\partial {t_h}}\bigg\vert_{t_h=t}=\oo(t^{p+1}),\qquad\quad\frac{\partial \tilde{D}^{t,t_h}_{i+1,1}}{\partial t_h}\bigg\vert_{t_h=t}=\oo(t^{p+1}),\\
			&\frac{\partial \tilde{D}^{t,t_h}_{i+1,j+1}}{\partial t_h}\bigg\vert_{t_h=t}=\oo(t),\label{dijf}
		\end{align}
	\end{subequations}
\end{lemma}
\begin{proof}
	We divide our proof in five steps.
	
	\noindent\textbf{Step 1}. We first verify the following identity which is essential to our proof. From Lemmas~\ref{cidi} and \ref{combination}, we have
	\begin{align}
		1+A_1=&\frac{\sum_{j=1}^{\lfloor\frac p2\rfloor +1}
			\Big\{{(-1)}^jt^{-2j}
			\Big[\frac{4(p+2j)!}{(p+2-2j)!(2j-1)!}+(-1)^p
			\frac{2(p+1)!}{(p+2-2j)!p^{p+2-2j}}\mathcal{N}\Big]\Big\}}{\sum_{j=1}^{ \lfloor\frac p2\rfloor +1}{(-1)}^jt^{-2j}\frac{4(p+2j)!}{(p+2-2j)!(2j-1)!}}\nonumber\\
		=&\frac{-t^{-2}\frac{4(2p+1)!}{p!p^p}}{\sum_{j=1}^{\lfloor\frac  p2\rfloor +1}{(-1)}^jt^{-2j}\frac{4(p+2j)!}{(p+2-2j)!(2j-1)!}}\nonumber\\
		=&(-1)^{N_e}\frac{(2p+1)!(2N_e+1)!(p-2N_e)!}{p!(p+2+2N_e)!}p^{-p}{t}^{2N_e}+\oo({t}^{2N_e+2})\nonumber\\
		=&\left\{
		\begin{array}{ll}
			\frac{\ii^p}{2} p^{-p}{t}^{p}+\oo({t}^{p+2}),&\qquad p\text{ even},\vspace{2mm}\\
			\ii^{p-1}p^{-p}{t}^{p-1}+\oo({t}^{p+1}),&\qquad p\text{ odd}.
		\end{array}\right.\label{1pa1}
	\end{align}
	
	\noindent\textbf{Step 2}. Next, we derive the expressions for $\tilde{\mathcal D}_{1,j+1}^{t,t_h}$ and $\tilde{\mathcal D}_{1,1}^{t,t_h}$ provided that  $p$ is even.  Noting that $B_t$ is symmetric, and $\xi_i(i=0,1)$ is orthogonal to $\lambda_i(1\leq i\leq p-1)$ (see \eqref{defcidi}), from  \eqref{Q}, Lemmas~\ref{1deq} and \ref{cidi}, we conclude that
	\begin{align}
		\tilde{\mathcal D}_{1,j+1}^{t,t_h}=&\mathcal{D}^{t,t_h}_{1,j+1}+\sum_{i=1}^{p-1}\left((c_i+e^{-\ii t_h}d_i)\mathcal{D}_{i+1,j+1}^{t,t_h}\right)\nonumber\\
		=&\bigg(B_t(\lambda_j,\lambda_0)+\sum_{i=1}^{p-1}c_iB_t(\lambda_j,\lambda_i)\bigg)+
		e^{-\ii t_h}\bigg(B_t(\lambda_{j},\lambda_p            )+\sum_{i=1}^{p-1}d_iB_t(\lambda_j,\lambda_i)\bigg)\nonumber\\
		&+{(-1)}^j\tbinom pjp^{2p}\Bigg((1+e^{-\ii t_h}-e^{\ii t_h}-e^{-2\ii t_h})\nonumber\\
		&+2(1-\cos t_h)\bigg(\sum_{i=1}^{p-1}{(-1)}^i c_i\tbinom pi+e^{-\ii t_h}\sum_{i=1}^{p-1}{(-1)}^i d_i\tbinom pi\bigg)\Bigg)\gamma\nonumber\\
		=&B_t(\lambda_j,\xi_0)+e^{-\ii t_h}B_t(\lambda_j,\xi_1)\nonumber\\
		&+2{(-1)}^j\tbinom pj(1-\cos t_h)\left(\frac{1+e^{-\ii t_h}-e^{\ii t_h}-e^{-2\ii t_h}}{2-2\cos t_h}+(1+e^{-\ii t_h})A_1\right)p^{2p}\gamma\nonumber\\
		=&2{(-1)}^j\tbinom pj(1-\cos t_h)(1+e^{-\ii t_h})\left(1+A_1\right)p^{2p}\gamma.\label{a1j}
	\end{align}
	According to Lemma~\ref{1deq} and the definition of $\tilde{\mathcal D}^{t,t_h}$, we could easily derive that
	
	\begin{align}\label{ai1}
		\tilde{\mathcal D}_{i+1,1}^{t,t_h}
		&=2{(-1)}^i\tbinom pi(1-\cos t_h)(1+e^{\ii t_h})\left(1+A_1\right)p^{2p}\gamma,\\
		\tilde{\mathcal D}_{i+1,j+1}^{t,t_h}&=B_t(\lambda_i,\lambda_j)+2{(-1)}^{i+j}\tbinom pi\tbinom pj (1-\cos t_h)p^{2p}\gamma. \label{aij}
	\end{align}
	
	From \eqref{Q}, \eqref{a11}, Lemmas~\ref{1deq},\ref{pepo} and \ref{cidi}, and the fact $B_t(\lambda_0,\lambda_0)=B_t(\lambda_p,\lambda_p)$, we have
	\begin{align}\label{a11}
		\tilde{\mathcal D}_{1,1}^{t,t_h}=& \mathcal{D}_{1,1}^{t,t_h}+\sum_{i=1}^{p-1}(c_i+e^{-\ii t_h}d_i)\mathcal{D}_{i+1,1}^{t,t_h}+\sum_{i=1}^{p-1}(c_i+e^{\ii t_h}d_i)\tilde{\mathcal D}_{1,i+1}^{t,t_h}\nonumber\\
		=&B_t(\lambda_0,\lambda_0)+\sum_{i=1}^{p-1}c_iB_t(\lambda_i,\lambda_0)+B_t(\lambda_p,\lambda_p)
		+\sum_{i=1}^{p-1}d_iB_t(\lambda_i,\lambda_p)\nonumber\\
		&+e^{\ii t_h}\Big(B_t(\lambda_0,\lambda_p)
		+\sum_{i=1}^{p-1}c_iB_t(\lambda_i,\lambda_p)\Big)+e^{-\ii t_h}\Big(B_t(\lambda_p,\lambda_0)
		+\sum_{i=1}^{p-1}d_iB_t(\lambda_i,\lambda_0)\Big)\nonumber\\
		&+p^{2p}\bigg(2(1-\cos 2t_h)+(1+e^{\ii t_h}-e^{-\ii t_h}-e^{2\ii t_h})\sum_{j=1}^{p-1}{(-1)}^j\tbinom pj(c_j+e^{-\ii t_h}d_j)\nonumber\\
		&+(2-2\cos t_h)(1+e^{-\ii t_h})(1+A_1)\sum_{j=1}^{p-1}{(-1)}^j\tbinom pj(c_j+e^{\ii t_h}d_j)\bigg)\gamma\nonumber\\
		=&B_t(\xi_0,\lambda_0)+B_t(\xi_1,\lambda_p)+e^{\ii t_h}B_t(\xi_0,\lambda_p)+e^{-\ii t_h}B_t(\xi_1,\lambda_0)+p^{2p}2(1-\cos {2t_h})\gamma\nonumber\\
		&\times\bigg(1+\frac{1+e^{\ii t_h}-e^{-\ii t_h}-e^{2\ii t_h}}{2-2\cos {2t_h}}(1+e^{-\ii t_h})A_1\nonumber\\
		&\quad+\frac{1-\cos {t_h}}{1-\cos {2t_h}}(1+e^{-\ii t_h})(1+e^{\ii t_h})(1+A_1)A_1\bigg)\nonumber\\
		=&B_t(\xi_0,\xi_0)+B_t(\xi_1,\xi_1)+2\cos t_hB_t(\xi_0,\xi_1)+p^{2p}2(1-\cos {2t_h}){(1+A_1)}^2\gamma\nonumber\\
		=&\frac{1+\cos t_h}{2}B_t(\pe,\pe)+
		\frac{1-\cos t_h}{2}B_t(\po,\po)+2p^{2p}(1-\cos {2t_h}){(1+A_1)}^2\gamma.
	\end{align}
	
	\noindent\textbf{Step 3}. Notice that $B_t(\pe,\pe)$, $B_t(\po,\po)$, $A_1$ are independent of $t_h$, thus it follows that
	\begin{align*}
		\frac{\partial \tilde{D}^{t,t_h}_{1,1}}{\partial t_h}\bigg\vert_{t_h=t} &=\frac12\sin tB_t(\po,\po)-\frac12\sin tB_t(\pe,\pe)+4p^{2p}\sin 2t{(1+A_1)}^2\gamma\\
		\frac{\partial \tilde{D}^{t,t_h}_{1,j+1}}{\partial t_h}\bigg\vert_{t_h=t}
		&=2(-1)^j\tbinom{p}{j} \ii (e^{-2\ii t}-\cos t)(1+A_1)p^{2p}\gamma,\\
		\frac{\partial \tilde{D}^{t,t_h}_{i+1,1}}{\partial t_h}\bigg\vert_{t_h=t}
		&=2(-1)^i\tbinom{p}{i} \ii (\cos t-e^{2\ii t})(1+A_1)p^{2p}\gamma,\\
		\frac{\partial \tilde{D}^{t,t_h}_{i+1,j+1}}{\partial t_h}\bigg\vert_{t_h=t}
		&=2{(-1)}^{i+j}\tbinom pi\tbinom pj \sin t p^{2p}\gamma.
	\end{align*}
	
	\noindent\textbf{Step 4}.  Next we complete the proofs of \eqref{dija}--\eqref{dijf} for even $p$. Using the identities $(1-\cos t)\cot{\frac t2}-(1+\cos t)\tan{\frac t2}=0$ and $\cot \frac{t}{2}+\tan \frac{t}{2}=2\csc t$, \eqref{btpepo} in Lemma~\ref{pepo}, \eqref{1pa1} in Step 1,  along with the Taylor expansions of some elementary functions (e.g. $\cos t =1-\frac 12 t^2+\oo(t^4)$), it follows from  Step 2 with $t_h=t$ and Step 3 that
	\begin{align*}
		\tilde{\mathcal D}_{1,1}^{t,t}=&-\frac12(1+\cos t)\bigg(2t\tan \frac{t}{2}-\left[\frac{(p+1)!}{(2p+2)!}\right]^2\frac{t^{2p+4}}{2p+3}+\oo(t^{2p+6})\bigg)\\
		&+\frac12(1-\cos t)\bigg(2t\cot \frac{t}{2}+4\left[\frac{p!}{(2p)!}\right]^2\frac{t^{2p}}{2p+1}+\oo(t^{2p+2})\bigg)\\
		&+2p^{2p}(1-\cos {2t}){(1+A_1)}^2\gamma\\
		=&\bigg(\frac{1}{2p+1}\left[\frac{p!}{(2p)!}\right]^2+\gamma\bigg)t^{2p+2}+\oo(t^{2p+4}).\\
		\tilde{\mathcal D}_{1,j+1}^{t,t}
		=&{(-1)}^j {\ii}^p \tbinom pj p^p \gamma t^{p+2}+\oo(t^{p+4})=\oo(t^{p+2}),\\
		\tilde{\mathcal D}_{i+1,1}^{t,t}
		=&{(-1)}^i {\ii}^p \tbinom pi p^p \gamma t^{p+2}+\oo(t^{p+4})=\oo(t^{p+2}),\\
		\tilde{\mathcal D}_{i+1,j+1}^{t,t}
		=&\int_0^1\lambda_i'\lambda_j'+\oo(t^2),\\
		\frac{\partial \tilde{D}^{t,t_h}_{1,1}}{\partial t_h}\bigg\vert&_{t_h=t} 
		=\frac 12 \sin t\left(2t\cot \frac t2+\oo(t^{2p})\right)\\
		&\qquad\;\;\,-\frac 12 \sin t\left(-2t\tan \frac t2+\oo(t^{2p+4})\right)+\oo(t^{2p+1})\\
		&\quad\;\;\,=2t+\oo(t^{2p+1}),\\
		\frac{\partial \tilde{D}^{t,t_h}_{1,j+1}}{\partial t_h}\bigg\vert&_{t_h=t}
		=2(-1)^j\tbinom pj (\ii p)^p\gamma t^{p+1}+\oo(t^{p+2})=\oo(t^{p+1}),\\
		\frac{\partial \tilde{D}^{t,t_h}_{i+1,1}}{\partial t_h}\bigg\vert&_{t_h=t}
		=2(-1)^i\tbinom pi (\ii p)^p\gamma t^{p+1}+\oo(t^{p+2})=\oo(t^{p+1}),\\
		\frac{\partial \tilde{D}^{t,t_h}_{i+1,j+1}}{\partial t_h}\bigg\vert&_{t_h=t}
		=2{(-1)}^{i+j}\tbinom pi\tbinom pj p^{2p}\gamma t+\oo(t^3)=\oo(t).
	\end{align*}
	
	\noindent\textbf{Step 5}. The cases for odd $p$ can be proved by following similar lines in Steps 2--4 and are omitted. 
	This completes the proof of the theorem.
\end{proof}

The following lemma gives some results of implicit function theorem, which will used to prove the existence of a discrete wave number near the exact wave number and to estimate the phase error.

\begin{lemma}\label{dispersion}
	Let $F$ be a binary continuous function on the rectangle $[0, b_1]\times [0,b_2]$ for some positive $b_1$ and $b_2$ with the following properties: 
	\begin{align}
		&F(s,s)=\delta_0s^{\sigma_1}+o(s^{\sigma_1}),\label{cond1}\\
		&F_1^{\prime}(s,s)=\delta_1 s^{\sigma_2}+o(s^{\sigma_2}),\quad\text{as } s\to 0^+,\label{cond2}\\
		&F_1'' \text{is continuous at }(0,0),
	\end{align}
	where  $\sigma_1 >2\sigma_2>0, \delta_0$, and $\delta_1\neq 0$ are some constants independent of $s$.  Then there exists a constant $s_0>0$ such that for any $s\in(0,s_0]$, there exists a $s_h>0$ such that
	\begin{align*}
		F(s_h,s)&=0\quad\text{and}\quad
		|s-s_h|=\Big(\Big|\frac{\delta_0}{\delta_1}\Big|+o(1)\Big)s^{\sigma_1-\sigma_2}.
	\end{align*}
\end{lemma}
\begin{proof}
	Without loss of generality, we assume that $\delta_1>0$. For $s>0$ sufficiently small, the Taylor series expansion of $\tilde{F}_s$ at $s$ gives
	\begin{align*}
		F\big(s+s^{\frac{\sigma_1}{2}},s\big)&=F(s,s)+F_1'(s,s)s^{\frac{\sigma_1}{2}}+\frac{1}{2}F_1''\big(s+\theta_1 s^{\frac{\sigma_1}{2}},s\big)s^{\sigma_1}\\
		&=\big(\delta_1+o(1)\big)s^{\frac{\sigma_1}{2}+\sigma_2}>0,\\
		F\big(s-s^{\frac{\sigma_1}{2}},s\big)&=F(s,s)-F_1'(s,s)s^{\frac{\sigma_1}{2}}+\frac{1}{2}F_1''\big(s-\theta_1 s^{\frac{\sigma_1}{2}},s\big)s^{\sigma_1}\\
		&=-\big(\delta_1+o(1)\big)s^{\frac{\sigma_1}{2}+\sigma_2}<0.
	\end{align*}
	where $\theta_1,\theta_2\in (0,1)$.
	Since $F$ is continuous, there exists $s_h\in\big(s-s^{\frac{\sigma_1}{2}},s+s^{\frac{\sigma_1}{2}}\big)$ such that
	\begin{align}
		F(s_h,s)&=0\quad\text{and}\quad |s-s_h|=o(s).\label{disp1}
	\end{align}
	
	Furthermore, \eqref{disp1} reveals that
	\begin{align*}
		0&=F(s_h,s)=F(s,s)+F_1'(s,s)(s_h-s)+\frac{1}{2}F_1''\big(s+\theta(s_h-s),s\big)(s_h-s)^2,
	\end{align*}
	for some $\theta\in(0,1)$. If $F_1''\big(s+\theta(s_h-s),s\big)=0$, we could directly derive that
	$$|s_h-s|=\left|\frac{F(s,s)}{F_1'(s,s)}\right|=\Big|\frac{\delta_0}{\delta_1}\Big|s^{\sigma_1-\sigma_2}+o(s^{\sigma_1-\sigma_2}).$$
	
	If $F_1''\big(s+\theta(s_h-s),s\big)\neq 0$, denote by $\delta_2:=F_1''(0,0)$ and by 
	$G:=F(s,s)F_1''\big(s+\theta(s_h-s),s\big)/\big(F_1'(s,s)\big)^2$. Noting that $G=\frac{\delta_0\delta_2}{\delta_1^2}s^{\sigma_1-2\sigma_2}+o(s^{\sigma_1-2\sigma_2})<\frac12$ for $s>0$ sufficiently small, from the quadratic formula, we come to
	\begin{align*}
		s_h^{\pm}-s&=\frac{-F_1'(s,s)\pm\sqrt{\big(F_1'(s,s)\big)^2-2F(s,s)F_1''\big(s+\theta(s_h-s),s\big)}}{F_1''\big(s+\theta(s_h-s),s\big)},\\
		s_h^+-s&=\frac{-2F(s,s)}{F_1'(s,s)+\sqrt{\big(F_1'(s,s)\big)^2-2F(s,s)F_1''\big(s+\theta(s_h-s),s\big)}}\\
		&=\frac{-2F(s,s)}{F_1'(s,s)\big(1+\sqrt{1-2G}\big)}=-\frac{\delta_0}{\delta_1}s^{\sigma_1-\sigma_2}+o(s^{\sigma_1-\sigma_2}).
	\end{align*}
	thus for $s>0$ sufficiently small, there exists a $s_h>0$, such that
	$|s_h-s|=\big|\frac{\delta_0}{\delta_1}\big|s^{\sigma_1-\sigma_2}+o(s^{\sigma_1-\sigma_2})$, the proof is then completed.
\end{proof}

We are now in the position to introduce the main result of this paper.
\begin{theorem}\label{main}
	When solving the one-dimensional Helmholtz equation \eqref{cip_fem} with $p^{\rm th}$ order CIP-FEM, there exists a constant $C_0>0$ such that if $\frac{kh}{p}\le C_0$, we have the following estimate for phase difference.	
	\begin{align*}
		|k-k_h|=\frac 12\bigg(\frac{1}{(2p+1)}\left[\frac{p!}{(2p)!}\right]^2+\gamma\bigg)k^{2p+1}h^{2p}+\oo(k^{2p+3}h^{2p+2}),
	\end{align*}
	
	As a consequence, taking the penalty parameter as $$\gamma_0=-\frac{1}{(2p+1)}\left[\frac{p!}{(2p)!}\right]^2$$
	can reduce the phase difference of CIP-FEM to $\oo(k^{2p+3}h^{2p+2})$.
	
\end{theorem}
\begin{proof}
	Let $\tilde t=\frac tp$, $\tilde t_h=\frac {t_h}{p}$, and
	$$\tilde{\mathcal F}(\tilde t_h,\tilde t,\gamma):= \mathcal F(t_h,t,\gamma):=\mathbf{Det}({\mathcal D}^{t,t_h}).$$
	notice that
	\begin{align}\label{Ftt}
		\tilde{\mathcal F}(\tilde t_h,\tilde t,\gamma)=\mathbf{Det}({\mathcal D}^{t,t_h})=\mathbf{Det}(Q^{t_h}{\mathcal D}^{t,t_h} (Q^{t_h})^\mathbf{H})=\mathbf{Det}(\tilde{\mathcal D}^{t,t_h}).
	\end{align}
	
	For any square matrix $A$, let $A_{i,j}^*$ denotes the  cofactor of the $(i,j)^{\rm{th}}$ entry of $A$, $\tilde D_0$ be the submatrix of $\tilde{\mathcal D}^{t,t_h}$ obtained by removing the first row and column of it.
	In order to apply Lemma~\ref{dispersion}, we need the following deductions. Applying the Laplace expansion for determinant, we have
	\begin{align*}
		\tilde{\mathcal F}(\tilde t_h,\tilde t,\gamma)=&\tilde{D}^{t,t_h}_{1,1}\left(\tilde D^{t,t_h}_{1,1}\right)^*+\sum_{j=1}^{p-1}\tilde{D}_{1,j+1}^{t,t_h}\big(\tilde{D}_{1,j+1}^{t,t_h}\big)^*\\
		=&\tilde{D}^{t,t_h}_{1,1} \left(\tilde D^{t,t_h}_{1,1}\right)^*-\sum_{j=1}^{p-1}\tilde{D}^{t,t_h}_{1,j+1}\sum_{i=1}^{p-1}\tilde{D}^{t,t_h}_{i+1,1}\big(\tilde D_0\big)_{i,j}^*,
	\end{align*}
	by the chain rule for derivative,
	\begin{align*}
		\tilde{\mathcal F}_1'(\tilde t_h,\tilde t,\gamma)=&\mathcal F_1'(t_h,t,\gamma)\frac{\partial t_h}{\partial \tilde t_h}\\
		=&p\Bigg(\frac{\partial \tilde{D}^{t,t_h}_{1,1}}{\partial t_h}\left(\tilde D^{t,t_h}_{1,1}\right)^*+\tilde{D}^{t,t_h}_{1,1}\frac{\partial \left(\tilde D^{t,t_h}_{1,1}\right)^*}{\partial t_h}-\sum_{i,j=1}^{p-1}\bigg(\frac{\partial \tilde{D}^{t,t_h}_{1,j+1}}{\partial t_h}\tilde{D}^{t,t_h}_{i+1,1}\big(\tilde D_0\big)_{i,j}^*\\
		&+\tilde{D}^{t,t_h}_{1,j+1}\frac{\partial \tilde{D}^{t,t_h}_{i+1,1}}{\partial t_h}\big(\tilde D_0\big)_{i,j}^*+\tilde{D}^{t,t_h}_{1,j+1}\tilde{D}^{t,t_h}_{i+1,1}\frac{\partial \big(\tilde D_0\big)_{i,j}^*}{\partial t_h}\bigg)\Bigg).
	\end{align*}
	
	From Lemmas~\ref{positive} and \ref{tildeD} and  the definition of determinant, we derive that
	
	\begin{align*}
		\tilde{\mathcal F}(\tilde t,\tilde t,\gamma)=&\bigg(\bigg(\frac{1}{2p+1}\left[\frac{p!}{(2p)!}\right]^2+\gamma\bigg)t^{2p+2}+\oo(t^{2p+4})\bigg)\left(\alpha_0+\oo(t^2)\right)+\oo(t^{2p+4})\\
		=&\alpha_0\bigg(\frac{1}{2p+1}\left[\frac{p!}{(2p)!}\right]^2+\gamma\bigg)p^{2p+2}\tilde t^{2p+2}+\oo(\tilde t^{2p+4}),\\
		\tilde{\mathcal F}_1'(\tilde t,\tilde t,\gamma)=&p\left(2t+\oo(t^{2p+1})\right)\left(\alpha_0+\oo(t^2)\right)+\oo(t^{2p+3})\\
		=&2\alpha_0p^2\tilde t+o(\tilde t).
	\end{align*}
	
	Taking a close observation of $\tilde{D}^{t,t_h}$,  we find that all the entries $\tilde{D}^{t,t_h}_{i,j}$ are $\mathbb C^\infty$ on $(t_h,t)$, so is $\tilde{\mathcal F}_1''$. Therefore, from Lemma ~\ref{dispersion},
	for $\tilde t$ sufficiently small, there exists $\tilde t_h=\frac{k_hh}{p}>0$ satisfying
	$$|k-k_h|=\frac ph|\tilde t_h-\tilde t|=\frac{1}{2}\bigg(\frac{1}{2p+1}\left[\frac{p!}{(2p)!}\right]^2+\gamma\bigg)k^{2p+1}h^{2p}+\oo(k^{2p+3}h^{2p+2}).$$
	this completes the proof of the theorem.
\end{proof}

\begin{remark}\label{remark_b}
	{\rm (a)} Taking $\gamma=0$ in 	Theorem~\ref{main}, the CIP-FEM degenerates to FEM and the corresponding phase error we deduced coincides with that of \cite{MA1}.

	{\rm (b)} Given $k,h,p$, let $\gamma^{opt}$ be the solution to \eqref{Ftt} in Theorem~\ref{main}, then $\gamma^{opt}$ is the optimal penalty parameter with which the CIP-FEM scheme \eqref{cip_fem} is pollution free in one-dimensional case, while an explicit form of $\gamma^{opt}$ for general $p$ is hard to find. For readers who may be interested, we list in Table~\ref{optprec} expressions of $\gamma^{opt}$ for $p=1,2,3,4$ (see also \cite{cip3}) and list in Table~\ref{opt} double-precision numerical approximations of $\gamma_0$ and $\gamma^{opt}$ for $\frac{kh}{p}=1$ for $p=1,2,\cdots,7$ as comparison.
	
	{\rm (c)} Adding more penalty terms (e.g. on jumps of normal derivatives ordered from $1$ to $p-1$) may also eliminate the pollution error for problems in 1D or further improve the pollution error    for problems in higher dimensions (see e.g. \cite{cip3}), while it is also hard to find explicit forms of the penalty parameters for general $p$.  
	
	{\rm (d)} In theoretical analysis, we require $\frac{kh}{p}$ to be sufficiently small. However, we'll see in Section 5 that the penalty parameter we derived in Theorem~\ref{main} behaves quite well under the assumption of $\frac{kh}{p}\approx 1$.
	
	{\rm (e)} Theorem~\ref{main} implies that the phase difference may be improved to be $|k-k_h|=\oo(k^{2p+3}h^{2p+2})$ if we take $\gamma=\gamma_0$. One may be interested in the coefficient (denoted by $c_p$) in this big $\oo$ term. In Table~\ref{cp}, we list $c_p$ for $p=1,2,\cdots, 8$, which is calculated by programming in MATLAB. We found that they obey the following formula
	\begin{align*}
		c_p=\frac{\gamma_0}{24}\left|\frac{r_p}{(2p+3)!!}-1\right|\quad\text{where}\quad r_p=\begin{cases}
			9&\text{if } p=1\\
			24(2p-3)!!&\text{if } p\geq2,
		\end{cases}
	\end{align*}
	and hence
	$$|k-k_h^{\gamma_0}|=\frac{\gamma_0}{24}\left|\frac{r_p}{(2p+3)!!}-1\right|k^{2p+3}h^{2p+2}+\oo(k^{2p+5}h^{2p+4}).$$
	
	We conjecture that the above formula holds also for general $p$, which has actually been verified for $p$ up to $13$ via MATLAB programming, although we can not prove it yet.
\end{remark}

\begin{table}[tbhp]
	{\footnotesize
		\caption{Explicit expressions of $\gamma^{opt}$ for CIP-FEM of order $p=1,2,3,4$.}\label{optprec}
		\begin{center}
			\begin{tabular}{|c|c|}\hline
				p&$\gamma^{opt}$\\\hline
				1&$\begin{matrix}
					\big((6+t^2) \cos(t)-6 +2t^2\big)\big/12\big/(1-\cos(t))^2 
				\end{matrix}$\\\hline
				2&$\begin{matrix}
					\big((240+ 16 t^2 + t^4 )\cos(t) 
					+ 104 t^2 - 3 t^4 - 240\big)\big/\big((960 t^2 - 11520)\cos(t) \\
					+ (5760 + 960 t^2 )\cos^2(t) - 1920 t^2 + 5760\big)
				\end{matrix}$\\\hline
				3&$\begin{matrix}
					\big((25200+ 1080t^2+ 30t^4+ t^6)\cos(t) + 11520t^2 - 540t^4+ 4t^6 - 25200\big)\big/ \\
					\big((36288000+ 2419200t^2+ 151200t^4)\cos^2(t) \\
					+(13305600t^2- 72576000- 604800t^4)\cos(t)- 15724800t^2+ 453600t^4 + 36288000\big)
				\end{matrix}$\\\hline
				4&$\begin{matrix}
					\big((5080320+ 161280t^2+ 3024t^4+ 48t^6+t^8)\cos(t) + 2378880t^2- 134064t^4+ 1800t^6- 5t^8 \\
					- 5080320\big)\big/ \big((1024192512000+ 43893964800t^2+ 1219276800t^4+ 40642560t^6)\cos^2(t)\\
					+(424308326400t^2- 2048385024000- 23166259200t^4+ 121927680t^6)\cos(t)  \\ - 468202291200t^2 
					+ 21946982400t^4- 162570240t^6 + 1024192512000\big)
				\end{matrix}$
				\\\hline
			\end{tabular}
		\end{center}
	}
\end{table}

\begin{table}[tbhp]
	{\footnotesize
		\caption{Penalty parameters for numerical experiments. }\label{opt}
		\begin{center}
			\begin{tabular}{|  c | c | c | }\hline
				$p$	&$\gamma=\gamma_0$&$\gamma=\gamma^{opt}\:(kh/p=1)$
				\\\hline
				1&$-8.333333333333333 \times 10^{-2}$&$-8.592096810583184\times 10^{-2}$\\	
				2&$ -1.388888888888889\times 10^{-3}$&$-1.758364973238755\times 10^{-3}$\\	
				3&$-9.920634920634921\times 10^{-6}$&$-1.896623966419027 \times 10^{-5}$\\	
				4&$ -3.936759889140842\times 10^{-8}$&$ -1.793840107031879 \times 10^{-7}$\\	
				5&$ -9.941312851365762\times 10^{-11}$&$ -1.642663180893377 \times 10^{-9}$\\	
				6&$ -1.737991757231777\times 10^{-13}$&$ -7.477550634563100 \times 10^{-11}$\\	
				7&$ -2.228194560553560\times 10^{-16}$&$ -2.132344906487912 \times 10^{-14}$\\  \hline
			\end{tabular}
		\end{center}
	}
\end{table}

\begin{table}[tbhp]
	{\footnotesize
		\caption{Coefficients of the leading terms for $|k-k_h^{\gamma_0}|$. }\label{cp}
		\begin{center}
			\begin{tabular}{  c  c  c  c  c c }\hline\hline
				$p$	&$c_p$&$p$&$c_p$&$\quad p$&$c_p$
				\\\hline
				$1$&$1/720$&$4$&$223/140826470400$&$7$&$1097/119020127189483520000$\\
				$2$&$1/22400$&$5$&$421/103567809945600$&$8$&$1607/177431543456322355200000$\\
				$3$&$97/254016000$&$6$&$101/14104949354496000$&$\cdots$&$\cdots$
				\\  \hline\hline
			\end{tabular}
		\end{center}
	}
\end{table}
\section{Extension to multi-dimensions}
In this section we will show that the Theorem~\ref{main} still holds in higher dimensions ($d=2,3$). We will sketch the proof for two dimensions and then explain how to generalize to three dimensions.

First, we recall the following definition and properties of Kronecker matrix product which are esssential to our investigation.
\begin{definition}[Kronecker Product \cite{graham2018kronecker,Kronecker}]
	If $Y=\left(y_{ij}\right)_{m\times n}$, $Z=\left(z_{ij}\right)_{q\times r}$, then the Kronecker product $Y\otimes Z$ is an $mq\times nr$ block matrix in the form of
	\begin{align*}
		Y\otimes Z=\left(\begin{array}{ccc}
			y_{11}Z&\cdots&y_{1n}Z\\
			\vdots&\ddots&\vdots\\
			y_{n1}Z&\cdots&y_{nn}Z
		\end{array}\right)
	\end{align*}
	{\rm Properties} 
	{\rm (a)} The Kronecker product is bilinear and associative: $X\otimes(Y+Z)=X\otimes Y+X\otimes Z$, $(X+Y)\otimes Z=X\otimes Z+Y\otimes Z$, $X\otimes(Y\otimes Z)=(X\otimes Y)\otimes Z$.
	
	{\rm (b)} If $Y_1$,$Y_2$, $Z_1$ and $Z_2$ are matrices of such size that can form the matrix products $Y_1Y_2$ and $Z_1Z_2$, we then have mixed-product property: $(Y_1\otimes Z_1)(Y_2\otimes Z_2)=Y_1Y_2\otimes Z_1Z_2$.
	
	{\rm (c)} Conjugate transposition is distributive over the Kronecker product: $(Y\otimes Z)^\mathbf{H}=Y^\mathbf{H}\otimes Z^\mathbf{H}$.
\end{definition}

We are now in the position to consider the two-dimensional case. Let $t_1=t\cos\theta,t_2=t\sin\theta,t_{h1}=t_h\cos\theta,t_{h2}=t_h\sin\theta$ where $t=kh,t_h=k_h h$ as before. The following lemma gives an explicit expression of   the coefficient matrix $\mathcal{D}$ in \eqref{equarray}, whose proof is not difficult but too long to give here and we leave it to Appendix~\ref{A4}.

\begin{lemma}\label{2deq}
	When solving the two-dimensional Helmholtz equation with $p^{\rm th}$ order CIP-FEM on the tensor product mesh, the coefficient matrix $\mathcal{D}$ associated to the set of generating nodes $\mathcal{X}_g=\{\bm{x}_{0,0},\bm{x}_{0,1},\cdots,\bm{x}_{p-1,p-1}\}$  (See Figure~\ref{2dgs}) takes the following form:
	\begin{align*}
		\mathcal{D}=\mathcal{D}^{t_1,t_{h1}}\otimes\mathcal{M}^{t_{h2}}+\mathcal{M}^{t_{h1}}\otimes\mathcal{D}^{t_2,t_{h2}}
	\end{align*}
	where $\mathcal{D}^{t_i,t_{hi}}$ is the coefficient matrix in Lemma~\ref{1deq} by replacing $t,t_h$ with $t_i,t_{hi}$, respectively, and ${\mathcal M}^\beta$ is defined by
	\begin{align*}
		\begin{array}{ll}
			\mathcal{M}_{1,1}^\beta&=2\int_0^1\lambda_0^2\D x+2\cos\beta\int_0^1\lambda_p\lambda_0\D x,\vspace{3mm}\\
			\mathcal{M}_{1,j+1}^\beta&=\int_0^1\lambda_j\lambda_0\D x+e^{-\ii \beta}\int_0^1\lambda_{p-j}\lambda_0\D x,\quad
			\mathcal{M}_{i+1,1}^\beta=\mathbf{conj}(\mathcal{M}_{1,i+1}^\beta),\vspace{3mm}\\
			\mathcal{M}_{i+1,j+1}^\beta&=\int_0^1\lambda_j\lambda_i\D x,\qquad 1\leq i,j\leq p-1.
		\end{array}
	\end{align*}
\end{lemma}

By analogy with the one-dimensional case, in order to calculate the determinant of matrix $\mathcal{D}$, we aim to transform it to a form which is more calculable. We need the following lemmas to proceed with our research whose rigorous proofs are postponed to Appendices~\ref{A5} and \ref{A6}.
\begin{lemma}\label{Mij}
	If $\beta=o(1)$ as $t\to 0$, 	$\tilde{\mathcal{M}}^\beta=Q^\beta\mathcal{M}^{\beta}(Q^\beta)^{\bm{H}}$ satisfies the following estimates:
	\begin{align*}
		\tilde{\mathcal{M}}_{1,1}^\beta&=1+o(1) \quad\text{and}\quad\tilde{\mathcal{M}}_{1,j+1}^\beta, \tilde{\mathcal{M}}_{j+1,1}^\beta=\int_0^1 \lambda_j\D x+o(1)\quad\forall 1\leq j\leq p-1,
	\end{align*}
	where $Q^\beta$ is defined in \eqref{Q}.
\end{lemma}

\sss

\begin{lemma}\label{2dpositive}
	Let $D_0$ be the matrix defined in Lemma~\ref{positive}. Denote by
	\begin{align*}
		\begin{array}{ll}
			D_1=
			\begin{pmatrix}
				0 &\bm{0}^T\\
				\bm{0}& D_0
			\end{pmatrix}
			%\left(\begin{array}{cccc}
				%0 & 0 & \cdots & 0 \\
				%0 & \int_0^1 \lambda_1' \lambda_1' \D x & \cdots &\int_0^1 \lambda_1' \lambda_{p-1}' \D x\\
				%\vdots & \vdots & \ddots & \vdots\\
				%0 & \int_0^1 \lambda_{p-1}' \lambda_1' \D x & \cdots &\int_0^1 \lambda_{p-1}' \lambda_{p-1}' \D x
				%\end{array}\right)
			,\quad
			M_1=\begin{pmatrix}
				1 & \int_0^1 \lambda_1 \D x & \cdots & \int_0^1\lambda_{p-1} \D x \\
				\int_0^1 \lambda_1 \D x & \int_0^1 \lambda_1 \lambda_1 \D x & \cdots &\int_0^1 \lambda_1 \lambda_{p-1} \D x\\
				\vdots & \vdots & \ddots & \vdots\\
				\int_0^1 \lambda_{p-1} \D x & \int_0^1 \lambda_{p-1} \lambda_1 \D x & \cdots &\int_0^1 \lambda_{p-1} \lambda_{p-1} \D x
			\end{pmatrix},
		\end{array}
	\end{align*}
	$$\hat{D}=D_1\otimes M_1+M_1\otimes D_1.$$		
	then we have $\hat{D}_{1,1}^*\neq 0$, where ${}^*$ denotes for the algebraic cofactor.
\end{lemma}

We remark here that Lemma~\ref{2dpositive} will play the role of Lemma~\ref{positive}. 

\begin{theorem}\label{2dmain}
	When solving the two-dimensional Helmholtz equation with $p^{\rm th}$ order CIP-FEM on the tensor product mesh in $\mathbb{R}^2$, there exists a constant $C_0$ such that if $\frac{kh}{p}\le C_0$, we have the following estimate for the phase difference.
	\begin{align*}	
		|k-k_h|=\frac 12\bigg(\frac{1}{(2p+1)}\left[\frac{p!}{(2p)!}\right]^2+\gamma\bigg)k^{2p+1}h^{2p}+\oo(k^{2p+3}h^{2p+2}),
	\end{align*}
	
	As a consequence, taking the penalty parameter as $$\gamma_0=-\frac{1}{(2p+1)}\left[\frac{p!}{(2p)!}\right]^2$$
	can reduce the phase difference of CIP-FEM to $\oo(k^{2p+3}h^{2p+2})$.
\end{theorem}
\begin{proof}
	Taking $Q=Q^{t_{h1}}\otimes Q^{t_{h2}}$, from Lemma~\ref{2deq}, we have
	\begin{align*}
		\tilde{D}&=Q\mathcal{D}Q^{\bm{H}}\\
		&=\Big(Q^{t_{h1}}\mathcal{D}^{t_1,t_{h1}}(Q^{t_{h1}})^{\bm{H}}\Big)\otimes \Big(Q^{t_{h2}}\mathcal{M}^{t_{h2}}(Q^{t_{h2}})^{\bm{H}}\Big)\\
		&+\Big(Q^{t_{h1}}\mathcal{M}^{t_{h1}}(Q^{t_{h1}})^{\bm{H}}\Big)\otimes \Big(Q^{t_{h2}}\mathcal{D}^{t_2,t_{h2}}(Q^{t_{h2}})^{\bm{H}}\Big)\\
		&=\tilde{\mathcal{D}}^{t_1,t_{h1}}\otimes \tilde{\mathcal{M}}^{t_{h2}}+ \tilde{\mathcal{M}}^{t_{h1}}\otimes \tilde{\mathcal{D}}^{t_2,t_{h2}}.
	\end{align*}
	
	Then it follows the same procedure as in the proof of Theorem~\ref{main}. Let $\tilde t=\frac tp$, $\tilde t_h=\frac {t_h}{p}$, and define
	$$\tilde{\mathcal F}(\tilde t_h,\tilde t,\gamma,\theta):=\mathcal F(t_h,t,\gamma,\theta):=\mathbf{Det}(\mathcal D)=\mathbf{Det}(\tilde{\mathcal D}).$$
	
	We only need to evaluate the leading term of $\tilde{\mathcal F}(\tilde t,\tilde t,\gamma,\theta)$ and $\tilde{\mathcal F}_{1}'(\tilde t,\tilde t,\gamma,\theta)$.
	Observing the structure of matrix $\tilde{D}$ and performing in a similar manner as that in Theorem~\ref{main}, we conclude that 
	\begin{align*}
		&\mathcal F(\tilde t,\tilde t,\gamma,\theta)\\
		=&\hat{D}_{1,1}^*\bigg(\frac{1}{(2p+1)}\left[\frac{p!}{(2p)!}\right]^2+\gamma\bigg)(t_1^{2p+2}+t_2^{2p+2})+\oo(t^{2p+4})\\
		=&\hat{D}_{1,1}^*\bigg(\frac{1}{(2p+1)}\left[\frac{p!}{(2p)!}\right]^2+\gamma\bigg)\left((\cos\theta)^{2p+2}+(\sin\theta)^{2p+2}\right)p^{2p+2}\tilde t^{2p+2}+\oo(\tilde t^{2p+4}),
	\end{align*}
	and
	$$\tilde{\mathcal F}_1'(\tilde t,\tilde t,\gamma,\theta)=2p\hat{D}_{1,1}^*(\cos^2\theta+\sin^2\theta)t+o(t)=2\hat{D}_{1,1}^*p^2\tilde t+o(\tilde t).$$
	
	Thus  by Lemma~\ref{dispersion},
	for $\tilde t$ sufficiently small, the phase difference in direction $\theta$ reads:
	$$\frac 12\bigg(\frac{1}{(2p+1)}\left[\frac{p!}{(2p)!}\right]^2+\gamma\bigg)\left((\cos\theta)^{2p+2}+(\sin\theta)^{2p+2}\right)k^{2p+1}h^{2p}+\oo(k^{2p+3}h^{2p+2}),$$
	notice that
	$$\left|(\cos\theta)^{2p+2}+(\sin\theta)^{2p+2}\right|\leq\left|(\cos\theta)^{2p+2}\right|+\left|(\sin\theta)^{2p+2}\right|\leq\cos^2\theta+\sin^2\theta=1,$$
	we finally obtain the conclusion as claimed.
\end{proof}
\begin{remark}
	The results in Theorem~\ref{2dmain} could be extended to the 3D case with
	$\mathcal{D}=\sum_{i=1}^{3}\Big(\big(\bigotimes_{j=1}^{i-1}\mathcal{M}^{t_{hj}}\big)\bigotimes\mathcal{D}^{t_i,t_{hi}}\bigotimes\big(\bigotimes_{j=i+1}^3\mathcal{M}^{t_{hj}}\big)\Big)$ and $Q=\bigotimes_{j=1}^{3}Q^{t_{hj}}$. We omit the details.
\end{remark}

\section{Numerical results}
In this section we will illustrate the pollution effect of the FEM, CIP-FEM with the penalty parameter $\gamma=\gamma_0=-\frac{1}{(2p+1)}\left[\frac{p!}{(2p)!}\right]^2$ we derived in Theorem~\ref{main} and \ref{2dmain} and the penalty parameter $\gamma=\gamma^{opt}$ as well. We also verify that the pollution term and  phase difference are of the same order.

According to the preasymptotic error analyses of FEM \cite{cip1,cip2,cip3,hp1,hp2}, the following error estimate holds for the finite element solution $u_h^{\rm FEM}$.
\eq{\label{errFEM}\|u-u_h^{\rm FEM}\|_{H^1}\lesssim k^ph^p+k^{2p+1}h^{2p}}
where the first term in the right hand side is the interpolation error and the second term is the pollution error which is of the same order as the phase difference. Since the phase difference of the CIP-FEM is of order $k^{2p+3}h^{2p+2}$ if $\gamma=\gamma_0$ (see Theorem~\ref{main} and \ref{2dmain}), we expect the following error estimate for the CIP-FEM with $\gamma=\gamma_0$
\eq{
	\label{errCIP}\|u-u_h\|_{H^1}\lesssim k^ph^p+k^{2p+3}h^{2p+2}}
which reduce the pollution error of the FEM  to $Ck^{2p+3}h^{2p+2}$.

\begin{example}\label{Ex1} We simulate the following one dimensional Helmholtz problem:
	\begin{equation*}
		\begin{cases}
			&-u^{\prime\prime}-k^2u=1,\quad \text{in } (0,1),\\
			&u(0)=0,\quad
			u^\prime(1)+\ii k u(1)=0,\\
		\end{cases}
	\end{equation*}
	whose exact solution reads:
	$u=\frac{1}{k^2}\big(e^{-\ii kx}+\ii e^{-\ii k}\sin(kx)\big).$
\end{example}

Figure~\ref{loglogh} presents log-log plots of relative $H^1$ errors and phase differences $|k-k_h|$ versus the reciprocal of mesh size $h$ for FEM and CIP-FEM with $p=1,2,3$, respectively. Note that the slope of the error curve is $-m$ means that the convergence order of the error in $h$ is $m$. For $k=10$, the convergence orders in $h$ of the FE and CIP-FE solutions are coincide with that of the best approximation (with convergence order $p$ in $h$), which indicates that no pollution effect occurs for small wave number $k$. As $k$ grows larger ($k=10^3$ and $10^4$), the convergence orders of both the pollution error and phase difference with respect to $h$ are $2p$  for the FEM and $2p+2$ for the CIP-FEM with $\gamma=\gamma_0$, respectively, while the CIP-FEM with $\gamma=\gamma^{opt}$ remains unpolluted. Notice that the pollution effect diminishes as $h$ becomes smaller and enters the  asymptotic regime.% which reveals that the pollution term could also be largely reduced by refining the mesh.
\begin{figure}
	\centering
	\includegraphics[scale=0.65,trim={2cm 0 2cm 0}]{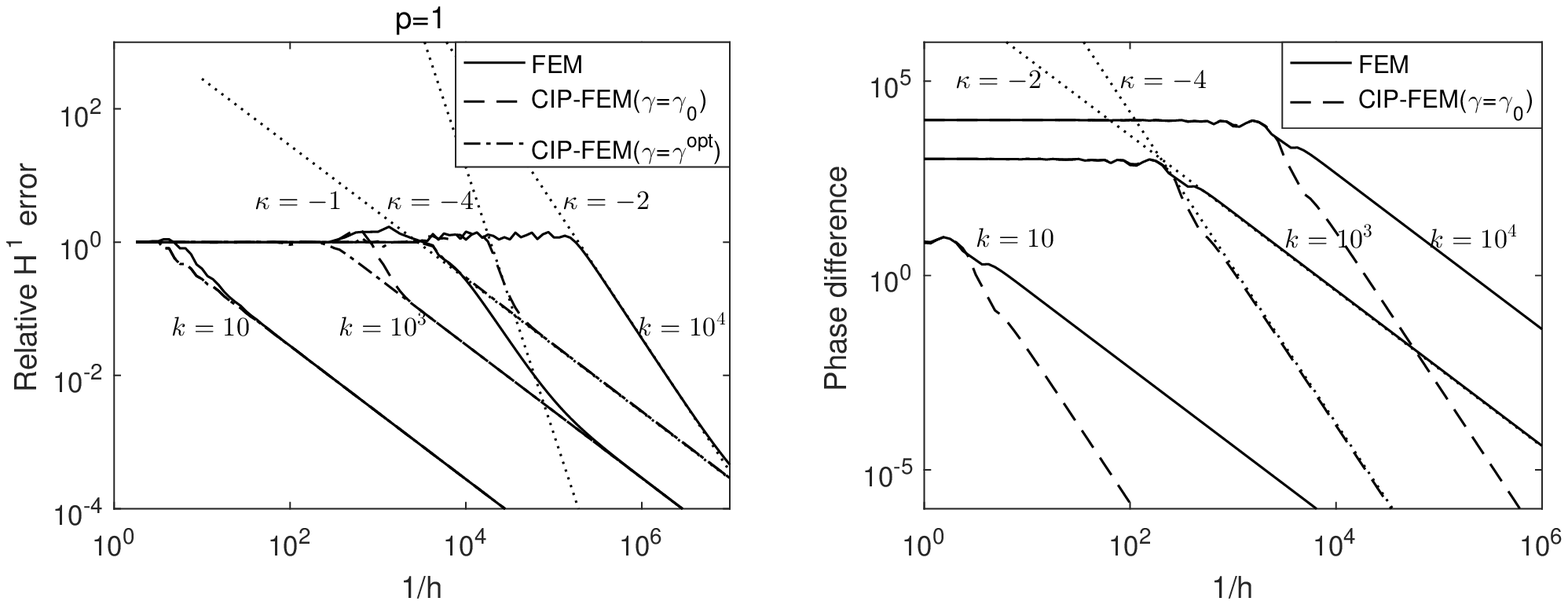}\\
	\includegraphics[scale=0.65,trim={2cm 0 2cm 0}]{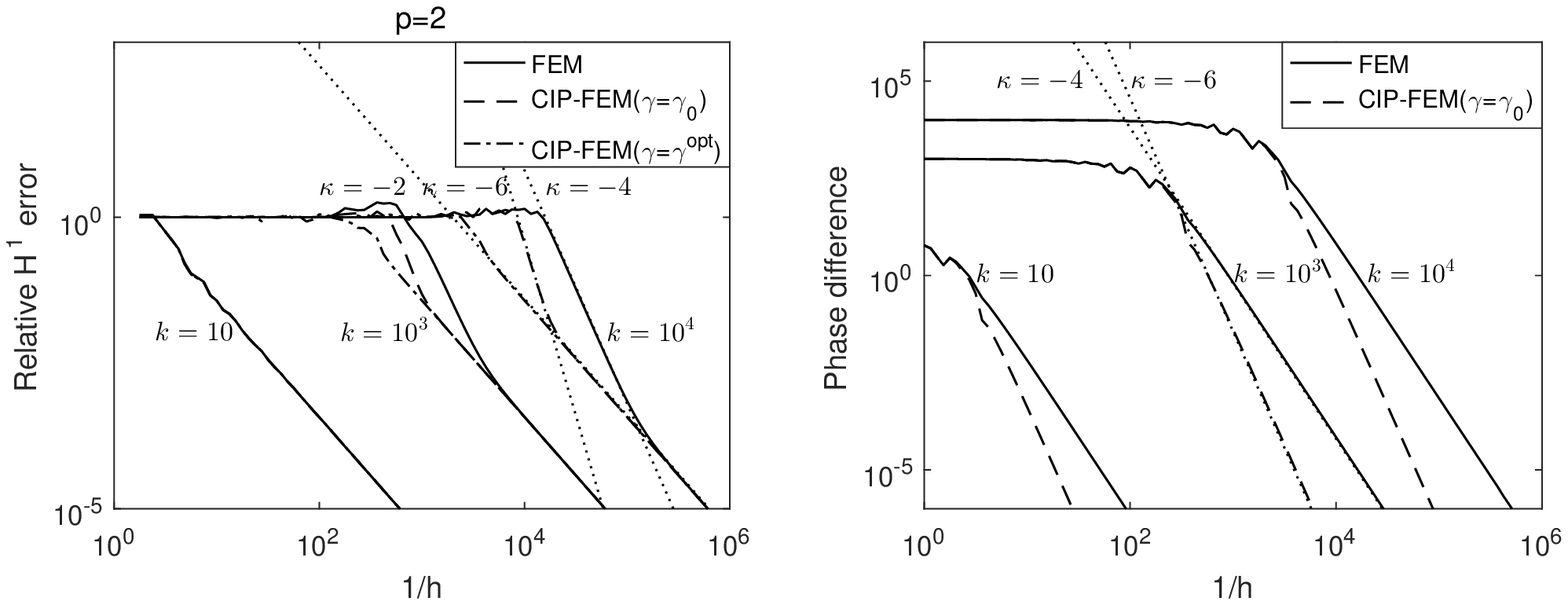}\\
	\includegraphics[scale=0.65,trim={2cm 0 2cm 0}]{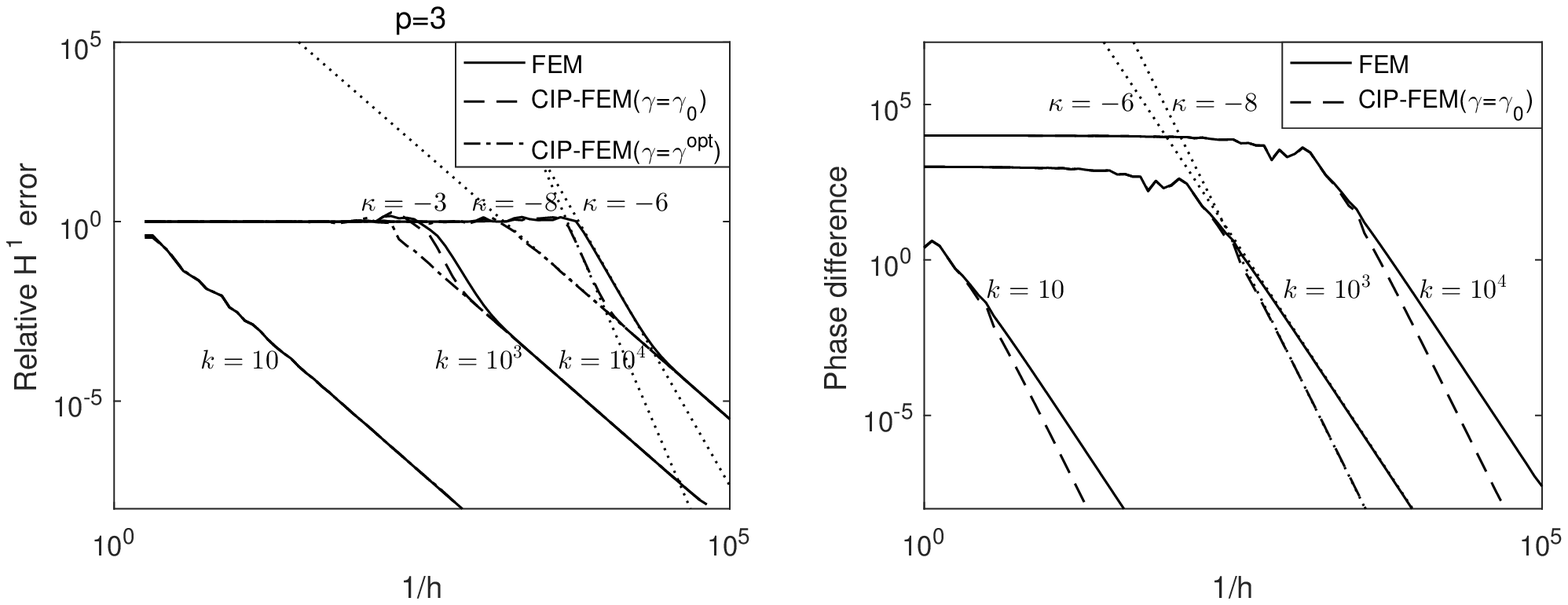}
	\caption{Example~\ref{Ex1}: Log-log plot of relative $H^1$ error (left) and phase difference (right) versus the reciprocal of the mesh size. The dotted lines give reference slopes denoted by $\kappa$.
		%	If $\frac{kh}{p}\approx 1$, the FEM, CIP-FEM (with $\gamma=\gamma_0$) and CIP-FEM (with $\gamma=\gamma^{opt}$) solutions are of $\oo(h^{2p})$, $\oo(h^{2p+2})$ and $\oo(h^p)$ respectively. 
		\label{loglogh}
	}
\end{figure}

Similar analysis could be applied on Figure~\ref{loglogk} that gives log-log plots of the errors versus the reciprocal of the wave number $k$, which especially verifies that the convergence orders of both the pollution error and phase difference with respect to $k$ are $2p+1$  for the FEM and $2p+3$ for the CIP-FEM with $\gamma=\gamma_0$, respectively, furthermore, taking $\gamma=\gamma^{opt}$ in CIP-FEM eliminates the pollution effect. These observations verify the error estimates in \eqref{errFEM} and \eqref{errCIP}.
\begin{figure}
	\centering
	\includegraphics[scale=0.65,trim={2cm 0 2cm 0}]{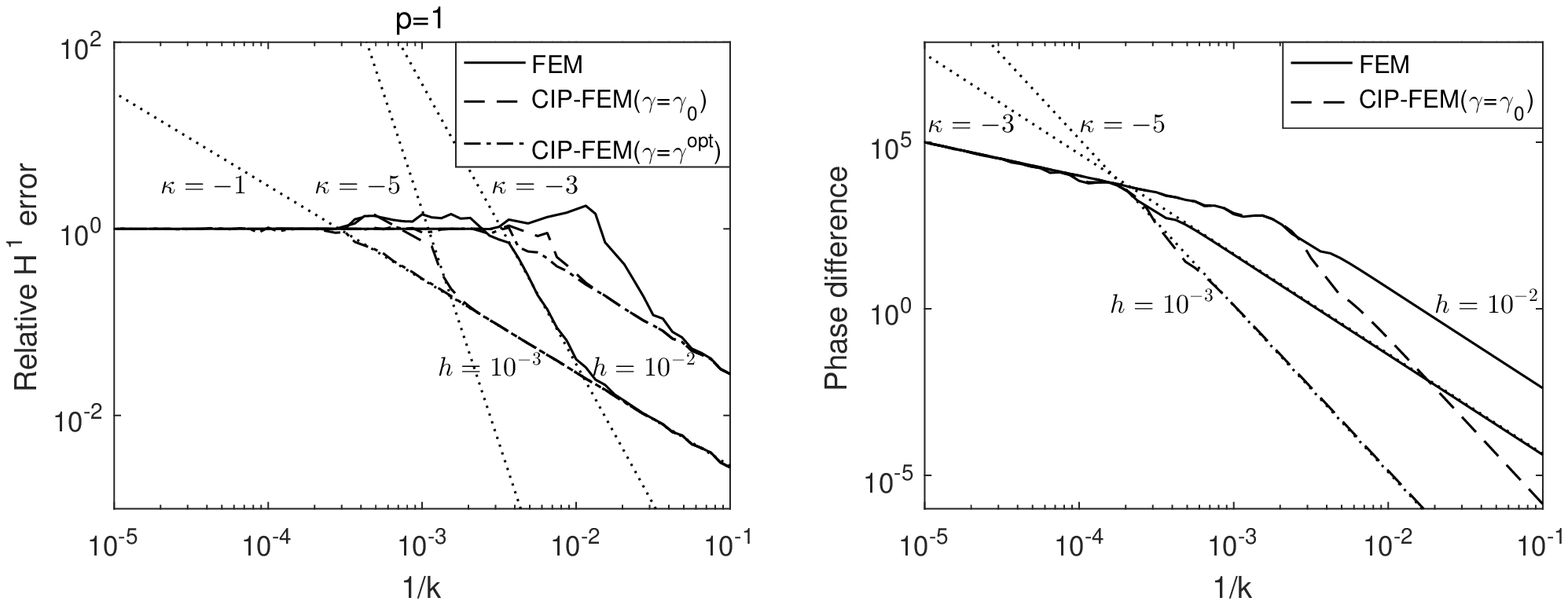}\\
	\includegraphics[scale=0.65,trim={2cm 0 2cm 0}]{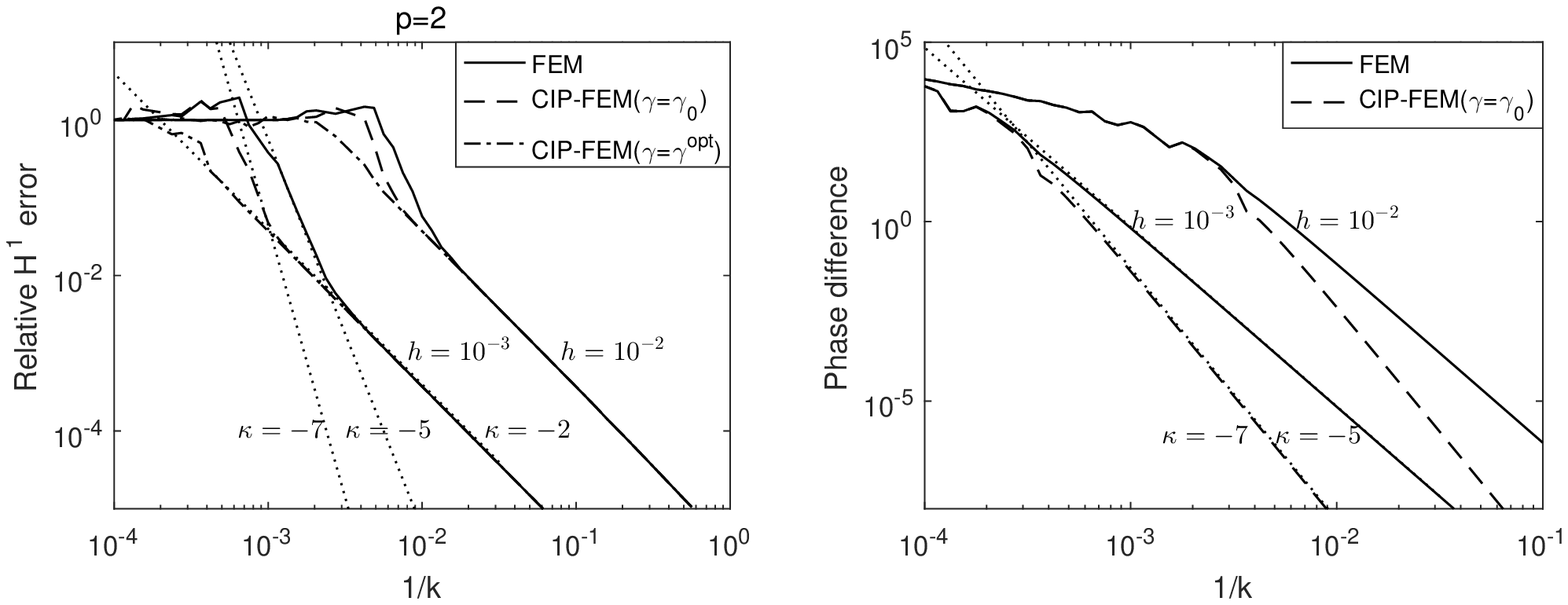}\\
	\includegraphics[scale=0.65,trim={2cm 0 2cm 0}]{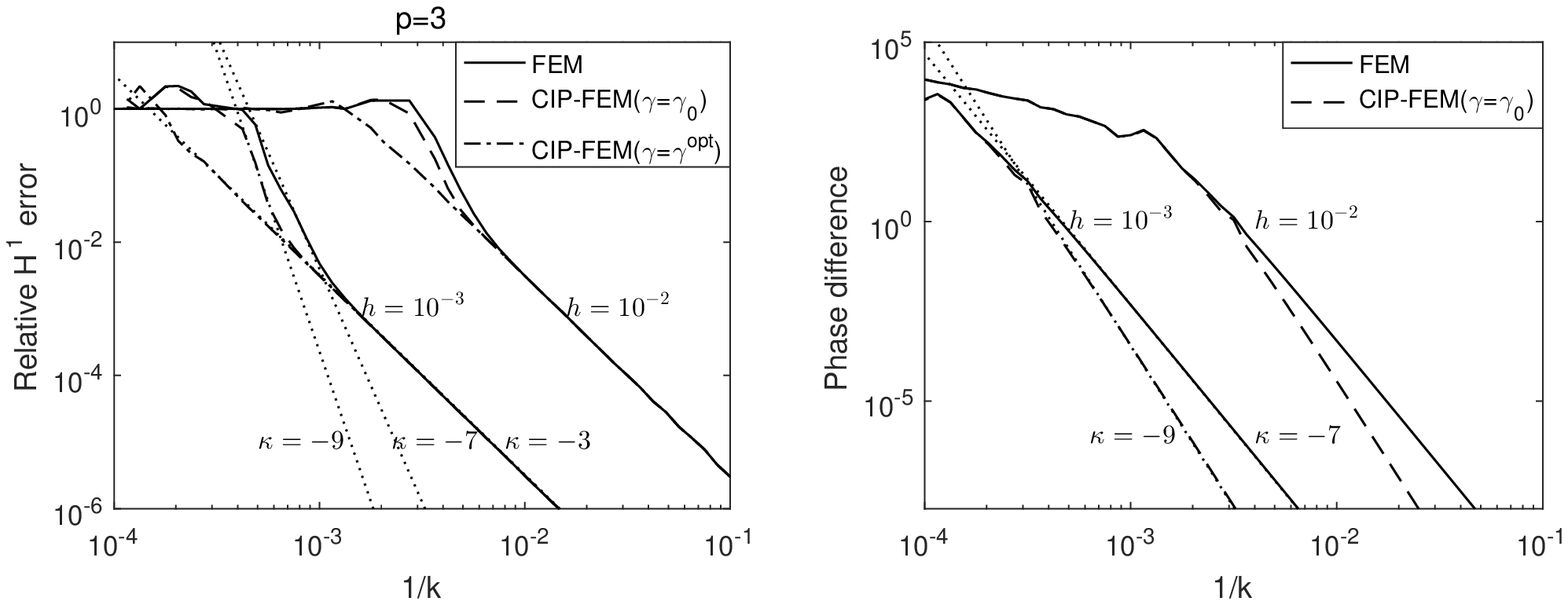}
	\caption{Example~\ref{Ex1}: Log-log plot of relative $H^1$ error (left) and phase difference (right) versus the reciprocal of the wave number. The dotted lines give reference slopes denoted by $\kappa$. 
		\label{loglogk}
	}
\end{figure}

\begin{example}\label{Ex2}
	We simulate the following two-dimensional Helmholtz equation:
	\begin{align*}
		\begin{cases}
			-\Delta u-k^2u=0,&in\;\;\Omega,\\
			\quad\frac{\partial u}{\partial\mathbf{n}}+\ii ku=g,&on\; \partial\Omega, 
		\end{cases}
	\end{align*}
	where $\Omega=(0,1)\times(0,1)$ and $g$ is so chosen depending on the exact solution $u=\sin\big(\frac{\sqrt{2}}{2}k(x+y)\big)$.
\end{example}

Figure~\ref{example} demonstrates the improvement of the CIP-FEM with $\gamma=\gamma_0$ compared with the FEM scheme intuitively. As it shown in the figure, By tuning the penalty parameter, the CIP-FEM can indeed significantly reduce the pollution error of the FEM.
\begin{figure}
	\centering
	\includegraphics[scale=0.7,trim={2cm 0cm 2cm 0cm}]{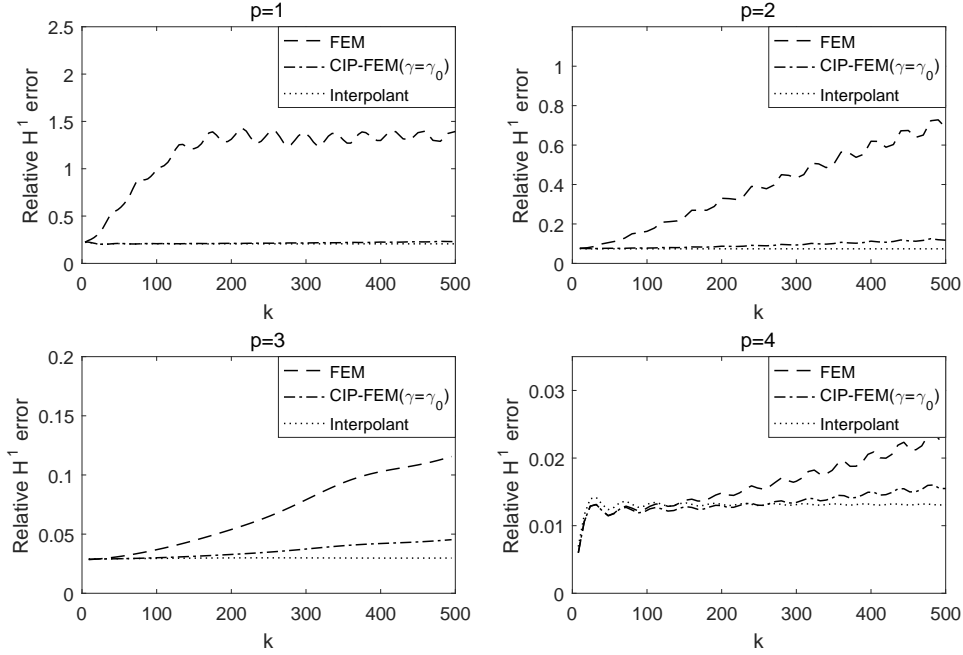}	
	\caption{Example~\ref{Ex2}: The relative errors of the FE solutions, the CIP-FE solutions ($\gamma=\gamma_0$), with mesh size $h$ determined by $\frac{kh}{p}=1$ for $p=1,2,3,4$ and $k=4,8,\cdots, 500$.\label{example}}
\end{figure}

Next we investigate the orders of the pollution errors. Due to the limitation of the computer, for this two-dimensional problem, we can only calculate the solutions on a mesh with mesh size as small as $h\approx 0.001$. Similar simulations as Figures~\ref{loglogh}and \ref{loglogk} in one-dimension can not show obvious convergence orders. We adopt the concept of ``critical mesh size" \cite{cip3} to verify the convergence orders of the pollution errors for the two-dimensional numerical example. 
\begin{definition}[Critical Mesh Size]
	Given a relative tolerance $\varepsilon$, a wave number $k$ and the degree of approximation space $p$, the critical mesh size $h(k, p, \varepsilon)$ with respect to the relative tolerance $\varepsilon$ is defined by the maximum mesh size such that the
	relative $H^1$ error of the (CIP-)FE solution is less than or equal to $\varepsilon$.
\end{definition}

Clearly, the critical mesh size is achieved in the preasymptotic regime and if the pollution error is $\oo(k^{m+1}h^m)$ for some positive integer $m$, then  $h(k, p, \varepsilon)=\varepsilon^\frac{1}{m}O(k^{-\frac{m+1}{m}})$.  Figure~\ref{criticalh} draws the log-log plot of the critical mesh sizes $h(k, p, \varepsilon)$ \cite{cip3} versus $k$ for FEM and CIP-FEM with $\gamma=\gamma_0$ for $p=1,2,3,4$ ($\varepsilon=0.5$ for $p=1,2$ and $\varepsilon=0.1$ for $p=3,4$). It is shown that the critical mesh sizes are of order   $k^{-\frac{2p+1}{2p}}$ for the FEM and $k^{-\frac{2p+3}{2p+1}}$ for the CIP-FEM with $\gamma=\gamma_0$, respectively, which supports the error estimates in \eqref{errFEM} and \eqref{errCIP}. 
\begin{figure}
	\centering
	\includegraphics[scale=0.7,trim={2cm 0cm 2cm 0cm}]{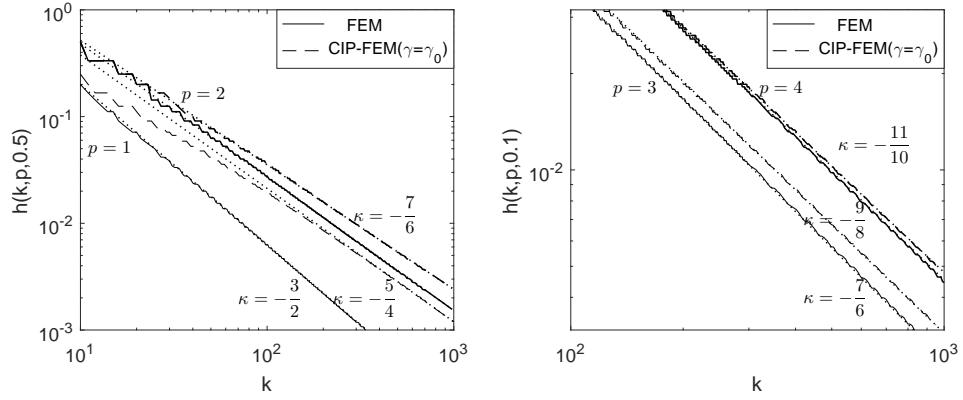}	
	\caption{Example~\ref{Ex2}: The log-log plot of the critical mesh size \cite{cip3} $h(k, p, \varepsilon)(p=1,2,3,4)$ versus $k$ for FEM and CIP-FEM with $\gamma=\gamma_0$ on 2D tensor product meshes.}
	\label{criticalh}
\end{figure}

\appendix
\section{Appendix} 

\subsection{Proof of Lemma~\ref{pepo}}\label{A1}
\begin{proof}
	Following \cite[\S 4.1]{MA1},
	define 
	\begin{align*}
		\hat{B}_\tau (u,v)&=\int_{-1}^1(u'v'-\tau^2uv)\D s,\\
		\pe^\tau(s)\;&=\frac{\sum_{j=1}^{N_e+1}{(-1)^j \tau^{-2j}{\mathcal L}_{2N_e+1}^{(2j-1)}(s)}}{\sum_{j=1}^{N_e+1}{(-1)}^j\tau^{-2j}{\mathcal L}_{2N_e+1}^{(2j-1)}(1)},\quad \po^\tau(s)\;=\frac{\sum_{j=1}^{N_o}{(-1)^j \tau^{-2j}{\mathcal L}_{2N_o}^{(2j-1)}(s)}}{\sum_{j=1}^{N_o}{(-1)}^j\tau^{-2j}{\mathcal L}_{2N_o}^{(2j-1)}(1)},
	\end{align*}
	where 
	$\mathcal{L}_n(s)$ is the Legendre polynomial of degree $n$ whose $d^{\rm th}$ derivative reads: 
	\begin{align*}
		\mathcal{L}_n^{(d)}(s)&=\frac{1}{2^nn!}\frac{\D ^{n+d}(s^2-1)^n}{\D s^{n+d}}\\
		&=\frac{n!}{2^n(n-d)!}\sum_{m=0}^{n-d}\tbinom{n+d}{m+d}\tbinom{n-d}{m}(s+1)^{n-m-d}(s-1)^m.
	\end{align*} 
	Inserting the equation above into \cite[(4.6)]{MA1} and simplifying, we obtain
	\begin{align*}
		\pe^\tau(s)&=\frac{1}{\sum_{j=1}^{N_e+1}\Big\{(-1)^j\tau^{-2j}{\frac{(2N_e+2j)!}{2^{2j-1}(2N_e+2-2j)!(2j-1)!}}\Big\}}\sum_{j=1}^{N_e+1}\bigg\{
		\frac{(-1)^j\tau^{-2j}(2N_e+1)!}{2^{2N_e+1}(2N_e+2-2j)!}\\
		&\sum_{m=0}^{2N_e+2-2j}\tbinom{2N_e+2j}{m-1+2j}\tbinom{2N_e+2-2j}{m}(s+1)^{2N_e+2-2j-m}(s-1)^m\bigg\},\\
		%%%%%%%%%%%%%%%%%%%%%%%%%%%%%%
		\po^\tau(s)&=\frac{1}{\sum_{j=1}^{N_o}\Big\{(-1)^j\tau^{-2j}{\frac{(2N_o-1+2j)!}{2^{2j-1}(2N_o+1-2j)!(2j-1)!}}\Big\}}\sum_{j=1}^{N_o}\bigg\{
		\frac{(-1)^j\tau^{-2j}(2N_o)!}{2^{2N_o}(2N_o+1-2j)!}\\
		&\sum_{m=0}^{2N_o+1-2j}\tbinom{2N_o-1+2j}{m-1+2j}\tbinom{2N_o+1-2j}{m}(s+1)^{2N_o+1-2j-m}(s-1)^m\bigg\}.
	\end{align*}
	
	It is obvious that
	$$B_t(u,v)=2\hat{B}_\frac{t}{2}(u,v),$$
	by taking $\tau=\frac{t}{2},x=\frac{s+1}{2}$ and combining with \cite[(4.15)--(4.16)]{MA1}, the proof is thus completed.
\end{proof}

\subsection{Proof of Lemma~\ref{cidi}}\label{A2}
\begin{proof}
	The representation of $c_i$ and $d_i$ could be proved readily from Lemma~\ref{pepo} owing to the uniqueness of $\xi_0$ and $\xi_1$. The remainder of the proof is straightforward.
	\begin{align*}
		A_{j}=\frac 1 2\sum_{i=1}^{p-1}{(-1)}^i \tbinom pi \pe\Big(\frac ip\Big)+\frac{(-1)^j}2\sum_{i=1}^{p-1}{(-1)}^i \tbinom pi \po\Big(\frac ip\Big),\quad i=1,2.
	\end{align*}
	
	If p  is even, inserting \eqref{exppepo} into the second summation of $A_1$, yields
	\begin{align}\label{wo1}
		& \sum_{i=1}^{p-1}{(-1)}^i \tbinom pi \po\Big(\frac ip\Big)\nonumber\\
		=&\frac{1}{\sum_{j=1}^{\frac p2 }\left\{(-1)^jt^{-2j}\frac{2(p-1+2j)!}{(p+1-2j)!(2j-1)!}\right\}}\sum_{j=1}^{\frac p2 }\bigg\{\frac{{(-1)}^jt^{-2j}2p!}{(p+1-2j)!p^{p+1-2j}}\nonumber\\
		&\sum_{i=1}^{p-1}{(-1)}^i \tbinom pi
		\sum_{m=0}^{p+1-2j} \tbinom{p-1+2j}{m-1+2j} \tbinom{p+1-2j}{m}i^{p+1-2j-m}(i-p)^m\bigg\}\nonumber\\
		=:&\frac{1}{\sum_{j=1}^{\frac p2 }\left\{(-1)^jt^{-2j}\frac{2(p-1+2j)!}{(p+1-2j)!(2j-1)!}\right\}}\sum_{j=1}^{\frac p2 }\bigg\{\frac{{(-1)}^jt^{-2j}2p!}{(p+1-2j)!p^{p+1-2j}}E_1\bigg\}.
	\end{align}
	
	By making the substitution of $i\rightarrow p-i,m\rightarrow p+1-2j-m$ in $E_1$, we have
	\begin{align*}
		E_1:=&\sum_{i=1}^{p-1}\sum_{m=0}^{p+1-2j}{(-1)}^i \tbinom pi \tbinom{p-1+2j}{m-1+2j} \tbinom{p+1-2j}{m}i^{p+1-2j-m}(i-p)^m\\
		=&\sum_{i=1}^{p-1}\sum_{m=0}^{p+1-2j}{(-1)}^{p-i} \tbinom p{p-i} \tbinom{p-1+2j}{p-m} \tbinom{p+1-2j}{p+1-2j-m}(p-i)^m(-i)^{p+1-2j-m}\\
		=&(-1)^{2p+1-2j}\sum_{i=1}^{p-1}\sum_{m=0}^{p+1-2j}{(-1)}^i \tbinom pi \tbinom{p-1+2j}{m-1+2j} \tbinom{p+1-2j}{m}(i-p)^mi^{p+1-2j-m}=-E_1.
	\end{align*}
	thus the second terms of $A_1$ and $A_2$ vanishes, we then arrive at
	\begin{align*}
		A_1=A_2=&\frac{1}{2}\sum_{i=1}^{p-1}{(-1)}^i \tbinom pi \pe\Big(\frac ip\Big)\\
		=&\frac{1}{\sum_{j=1}^{\frac p2+1}{(-1)}^jt^{-2j}\frac{4(p+2j)!}{(p+2-2j)!(2j-1)!}}\sum_{j=1}^{\frac p2+1}\bigg\{{(-1)}^jt^{-2j}
		\frac{2(p+1)!}{(p+2-2j)!p^{p+2-2j}}\vspace{1mm}\\
		&\sum_{i=1}^{p-1}{(-1)}^i \tbinom pi\sum_{m=0}^{p+2-2j} \tbinom{p+2j}{m-1+2j} \tbinom{p+2-2j}{m}
		i^{p+2-2j-m}(i-p)^m\bigg\}.
	\end{align*}
	
	If p  is odd, follow a similar argument, the details of which are omitted here, we obtain
	\begin{align*}
		A_1=-A_2=&\frac{1}{2}\sum_{i=1}^{p-1}{(-1)}^i \tbinom pi \po\Big(\frac ip\Big)\\
		=&\frac{-1}{\sum_{j=1}^{\frac{p+1}{2}}{(-1)}^jt^{-2j}\frac{4(p+2j)!}{(p+2-2j)!(2j-1)!}}\sum_{j=1}^{\frac{p+1}{2}}\bigg\{{(-1)}^jt^{-2j}
		\frac{2(p+1)!}{(p+2-2j)!p^{p+2-2j}}\vspace{1mm}\\
		&\sum_{i=1}^{p-1}{(-1)}^i \tbinom pi\sum_{m=0}^{p+2-2j} \tbinom{p+2j}{m-1+2j} \tbinom{p+2-2j}{m}
		i^{p+2-2j-m}(i-p)^m\bigg\}.
	\end{align*}
	This completes the proof of the lemma. 
\end{proof}
\subsection{Proof of Lemma~\ref{combination}}\label{A3}
\begin{proof}
	We will use $(0.154)_3$,$(0.154)_4$,$(0.156)_1$ in \cite{table} to complete the proof. 
	Set $0^0\equiv 1,\tbinom 00\equiv 1$. We first prove the equation below
	\begin{align}\label{E}
		E_2:=&\sum_{m=0}^p\tbinom{p+2}{m+1}\tbinom pm\nonumber\\
		=&\sum_{m=0}^p\tbinom{p+1}{m}\tbinom pm+\sum_{m=0}^p\tbinom{p+1}{m+1}\tbinom pm\nonumber\\
		=&\sum_{m=0}^p\tbinom{p+1}{m}\tbinom pm+\sum_{m=0}^p\tbinom{p+1}{p-m}\tbinom p{p-m}\nonumber\\
		=&\sum_{m=0}^p\tbinom{p+1}{m}\tbinom pm+\sum_{m=0}^p\tbinom{p+1}{m}\tbinom pm\nonumber\\
		=&2\sum_{m=0}^p\tbinom{p+1}{m}\tbinom p{p-m}
		\stackrel{(0.156)_1}{=}2\tbinom{2p+1}p.
	\end{align}
	
	By exchanging the order of summation and applying the binomial expansion to $\mathcal N$, we have
	\begin{align}\label{N}
		\mathcal N \! =\! &\!\sum_{i=0}^{p}{(-1)}^i\tbinom{p}{i}\sum_{m=0}^{p+2-2j}\tbinom{p+2j}{m-1+2j}\tbinom{p+2-2j}{m}
		i^{p+2-2j-m}{(i-p)}^m-2(-1)^p\tbinom{p+2j}{2j-1}p^{p+2-2j}\nonumber\\
		\! =\!&\!\sum_{m=0}^{p+2-2j}\!\tbinom{p+2j}{m-1+2j}\tbinom{p+2-2j}{m}\!\sum_{i=0}^{p}{(\!-\!1)}^i\tbinom{p}{i}
		i^{p+2-2j-m}\sum_{l=0}^m \tbinom ml i^l {(\!-\! p)}^{m-l}\!-\!2(\!-\!1)^p\tbinom{p+2j}{2j-1}p^{p+2-2j}\nonumber\\
		\! =\!&\!\sum_{m=0}^{p+2-2j}\tbinom{p+2j}{m-1+2j}\tbinom{p+2-2j}{m}T_{m,j}-2(-1)^p\tbinom{p+2j}{2j-1}p^{p+2-2j}
	\end{align}
	where
	\begin{align*}
		T_{m,j}=&\sum_{l=0}^m \tbinom ml {(-p)}^{m-l}\sum_{i=0}^{p}{(-1)}^i\tbinom{p}{i}
		i^{p+2-2j-m+l}
	\end{align*}
	for $m=0,j=1$,
	\begin{align*}
		T_{0,1}=\sum_{i=0}^{p}{(-1)}^i\tbinom{p}{i}
		i^p\stackrel{(0.154)_4}{=}(-1)^pp!
	\end{align*}
	for $1\leq m\leq p,j=1$, if $0\leq l\leq m-1$, we have $p-m+l\leq p-1$, thus
	\begin{align*}
		T_{m,1}=&\sum_{l=0}^{m-1} \tbinom ml {(-p)}^{m-l}\sum_{i=0}^{p}{(-1)}^i\tbinom{p}{i}
		i^{p-m+l}+\sum_{i=0}^{p}{(-1)}^i\tbinom{p}{i}
		i^p\\
		=&\sum_{i=0}^{p}{(-1)}^i\tbinom{p}{i}
		i^p=(-1)^pp!
	\end{align*}
	for $2\leq j\leq \lfloor{\frac p 2}\rfloor+1$,  $p+2-2j-m+l\leq p+2-2j\leq p-2$, hence
	$$T_{m,j}\stackrel{(0.154)_3}{=}0$$
	
	Combining with \eqref{E} and \eqref{N}, we obtain
	\begin{equation*}\mathcal N=\left\{
		\begin{array}{l}
			(-1)^pp!E_2-2(-1)^p(p+2)p^p=2{(-1)}^p \left(\frac{(2p+1)!}{(p+1)!}-(p+2)p^p\right), j=1\\
			2{(-1)}^{p+1}\tbinom{p+2j}{2j-1}p^{p+2-2j},\hskip 119pt 2\leq j\leq \lfloor{\frac p 2}\rfloor+1.
		\end{array}\right.
	\end{equation*}
	The proof is completed. 
\end{proof}

\subsection{Proof of Lemma~\ref{2deq}}\label{A4}
\begin{proof}
	For two-dimensional problem, $\{\xx_g\}=\{\xx_{0,0},\xx_{0,1},\cdots,\xx_{p-1,p-1}\}$ is a generating set of the mesh (see Figure~\ref{2dgs}).
	
re	Let $\phi_{i,j}=\phi_i(x)\phi_j(y)$ be the nodal basis function of the two-dimensional $\mathbb{Q}^p$ finite element space  at $\xx_{i,j}$, where $\phi_i=\lambda_i(\frac{x}{h})$ is the one-dimensional nodal basis function at  $x_i$. Similar to \eqref{em1}, we have for $ 0\le i_1,i_2,j_1,j_2 \le p$,
	\begin{align}\label{em2d1}
		&\int_0^h\int_0^h(\nabla\phi_{i_1,j_1}\cdot \nabla\phi_{i_2,j_2}-k^2\phi_{i_1,j_1}\phi_{i_2,j_2})\dd\xx\notag\\
		=&\int_0^h\int_0^h\Big(\big(\phi_{i_1}'(x)\phi_{i_2}'(x)-k^2\phi_{i_1}(x)\phi_{i_2}(x)\big)\phi_{j_1}(y)\phi_{j_2}(y)\notag\\
		&+\phi_{i_1}(x)\phi_{i_2}(x)\big(\phi_{j_1}'(y)\phi_{j_2}'(y)-k^2\phi_{j_1}(y)\phi_{j_2}(y)\big)\notag\\
		&+k^2\phi_{i_1}(x)\phi_{i_2}(x)\phi_{j_1}(y)\phi_{j_2}(y)\Big){\rm d} x{\rm d} y\notag\\
		=&B_t(\lambda_{i_1},\lambda_{i_2})\int_0^1\lambda_{j_1}\lambda_{j_2}{\rm d} x \! +\! B_t(\lambda_{j_1},\lambda_{j_2})\int_0^1\lambda_{i_1}\lambda_{i_2}{\rm d} x
		\!+\! t^2\int_0^1\lambda_{i_1}\lambda_{i_2}{\rm d} x\int_0^1\lambda_{j_1}\lambda_{j_2}{\rm d} x.
	\end{align}
	
	The lemma may be proved by writing equations at the nodal points $\xx_{0,0},\xx_{0,1},\cdots,$ $\xx_{p-1,p-1}$ and using \eqref{em2d1},\eqref{em2}, and  Lemma~\ref{1deq} to simplify the coefficients. The calculations are basic but quite tedious.  We take the $(1,1)^{\rm th}$ entry of $\DD$ as an example and omit the derivations of other entries.  Clearly, $\DD_{1,1}$ is the coefficient of $U_{0,0}$ in the equation at $\xx_{0,0}$, which is expressed as follows:
	\begin{align*}
		\DD_{1,1}=&\Aga(\phi_{0,0},\phi_{0,0})+2\cos(t_{h1})\Aga(\phi_{p,0},\phi_{0,0})+2\cos(t_{h2})\Aga(\phi_{0,p},\phi_{0,0})\\
		&+4\cos(t_{h1})\cos(t_{h2})\Aga(\phi_{p,p},\phi_{0,0})\\
		&+2\cos(2t_{h1})\Aga(\phi_{2p,0},\phi_{0,0})+2\cos(2t_{h2})\Aga(\phi_{0,2p},\phi_{0,0})\\
		&+4\cos(t_{h1})\cos(2t_{h2})\Aga(\phi_{p,2p},\phi_{0,0})+4\cos(2t_{h1})\cos(t_{h2})\Aga(\phi_{2p,p},\phi_{0,0}).
	\end{align*}
	From \eqref{em2d1},\eqref{em2}, and  Lemma~\ref{1deq}, we conclude that
	\begin{align*}
		\DD_{1,1}=&\Big[2B_{t_1}(\lambda_0,\lambda_0)+2\cos(t_{h1})B_{t_1}(\lambda_p,\lambda_0)\\
		&+\Big(2+(1-(-1)^p)^2-2(1-(-1)^p)^2\cos(t_{h1})-2(-1)^p\cos(2t_{h1})\Big)p^{2p}\gamma\Big]\\
		&\times\bigg(2\int_0^1\lambda_0^2+2\cos(t_{h2})\int_0^1\lambda_0\lambda_p\bigg)+\bigg(2\int_0^1\lambda_0^2+2\cos(t_{h1})\int_0^1\lambda_0\lambda_p\bigg)\\
		&\times\Big[2B_{t_2}(\lambda_0,\lambda_0)+2\cos(t_{h2})B_{t_2}(\lambda_p,\lambda_0)\\
		&+\Big(2+(1-(-1)^p)^2-2(1-(-1)^p)^2\cos(t_{h2})-2(-1)^p\cos(2t_{h2})\Big)p^{2p}\gamma\Big]\\
		=&\mathcal{D}_{1,1}^{t_1,t_{h1}}\otimes\mathcal{M}_{1,1}^{t_{h2}}+\mathcal{M}_{1,1}^{t_{h1}}\otimes\mathcal{D}_{1,1}^{t_2,t_{h2}}.
	\end{align*}
	This completes the proof.
\end{proof}
\subsection{Proof of Lemma~\ref{Mij}}\label{A5}
\begin{proof}
	We first prove the following estimates for a fixed $x$:
	\begin{align}\label{corollary}
		\pe(x)=1+\oo(t^2),\quad\po(x)=(2x-1)+\oo(t^2),\quad\text{as } t\rightarrow 0.
	\end{align}
	
	The proof of \eqref{corollary} relies on the following fact whose proof is trival and we omit it here,
	$$\frac{\sum_{i=1}^{m}a_i t^{-2i}}{\sum_{i=1}^{n}b_i t^{-2i}}=\frac{a_m}{b_n}t^{2n-2m}+\oo(t^{2n-2m+2}),\quad \text{ if }a_m,b_n\neq 0,m\leq n.$$
	thus an easy induction gives
	\begin{align*}
		\pe(x)&=\frac{(2N_e+1)!\tbinom{4N_e+2}{2N_e+1}}{(4N_e+2)!/(2N_e+1)!}+\oo(t^2)=1+\oo(t^2),\\
		\po(x)&=\frac{(2N_o)!\left(\tbinom{4N_o-1}{2N_o}(x-1)+\tbinom{4N_o-1}{2N_o-1}x\right)}{(4N_o-1)!/(2N_o-1)!}+\oo(t^2)=2x-1+\oo(t^2).
	\end{align*}
	
	According to Lemma~\ref{cidi} and \eqref{corollary}, we have
	$$c_i,d_i=\oo(1),\quad c_i+d_i=\pe(i/p)=1+\oo(t^2).$$ 
	
	The remainder of this proof is quite straightforward.
	\begin{align*}
		\tilde{\mathcal{M}}_{1,j+1}^\beta=&{\mathcal{M}}_{1,j+1}^\beta+\sum_{i=1}^{p-1}\big(c_i+d_i+(e^{-\ii \beta}-1)d_i\big)\mathcal{M}_{i+1,j+1}^\beta\\
		=&\int_0^1\lambda_j\lambda_0\D x+(1+\oo(\beta))\int_0^1\lambda_j\lambda_p\D x+\sum_{i=1}^{p-1}\big(1+o(1)\big)\int_0^1\lambda_j\lambda_i\D x\\
		=&\int_0^1\lambda_j\sum_{i=0}^{p}\lambda_i\D x+o(1)=\int_0^1\lambda_j\D x+o(1),\\
		\tilde{\mathcal{M}}_{1,1}^\beta\;=&{\mathcal{M}}_{1,1}^\beta+\sum_{i=1}^{p-1}(c_i+e^{-\ii \beta}d_i)\mathcal{M}_{i+1,1}^\beta+\sum_{j=1}^{p-1}(c_j+e^{\ii \beta}d_j)\tilde{\mathcal{M}}_{1,j+1}^\beta\\
		=&\int_0^1\lambda_0\Big(\lambda_0+\sum_{i=1}^{p-1}c_i\lambda_i\Big)\D x+e^{-\ii \beta}\int_0^1\lambda_0\Big(\lambda_p+\sum_{i=1}^{p-1}d_i\lambda_i\Big)\D x\\
		&+\int_0^1\lambda_p\Big(\lambda_p+\sum_{i=1}^{p-1}d_i\lambda_i\Big)\D x+e^{\ii \beta}\int_0^1\lambda_p\Big(\lambda_0+\sum_{i=1}^{p-1}c_i\lambda_i\Big)\D x\\
		&+\sum_{j=1}^{p-1}(c_j+e^{\ii \beta}d_j)\Big(\int_0^1\lambda_j\D x+o(1)\Big)\\
		=&\int_0^1(\lambda_0+\lambda_p)\Big(\lambda_0+\lambda_p+\sum_{i=1}^{p-1}(c_i+d_i)\lambda_i\Big)\D x\\
		&+\sum_{j=1}^{p-1}\big(1+o(1)\big)\Big(\int_0^1\lambda_j\D x+o(1)\Big)+\oo(\beta)\\
		=&\int_0^1(\lambda_0+\lambda_p)\big(1+o(1)\big)\D x+\sum_{j=1}^{p-1}\int_0^1\lambda_j\D x+o(1)=1+o(1).
	\end{align*}
	This completes the proof of the lemma.
\end{proof}

\subsection{Proof of Lemma~\ref{2dpositive}}\label{A6}
\begin{proof}
	Similar to the proof of Lemma~\ref{positive}, for any $\bm{v}=(v_0,\cdots,v_{p-1})^T$ $\in\mathbb{R}^p$, let
	$v=v_0+\sum_{i=1}^{p-1}v_i\lambda_i$, 
	we have 
	\begin{align*}
		\bm{v}^TM_1\bm{v}&=v_0^2+2\sum_{i=1}^{p-1}\int_0^1v_0v_i\lambda_i\D x+\sum_{i,j=1}^{p-1}\int_0^1v_i\lambda_iv_j\lambda_j\D x=\|v\|_{L_2(0,1)}^2\\
		&\gtrsim v_0^2+(v_0+v_1)^2+\cdots+(v_0+v_{p-1})^2+v_0^2\gtrsim |\bm v|^2, 
	\end{align*}
	where we have used the  inequality $(v_0+v_j)^2\ge \frac{1}{p+1} v_j^2-\frac1p v_0^2$ with   $\varepsilon=\frac{p}{p+1}$ to derive the last inequality. The above estimate says that $M_1$ is symmetric and positive definite (SPD).
	
	Since $D_0$ is SPD (see Lemma~\ref{positive}), there exists a non-singular $(p-1)\times(p-1)$ matrix $X$, such that 
	\begin{align*}
		XD_0 X^T=\mathbf{I}_{p-1},
	\end{align*}
	where $\mathbf{I}_{p-1}$ is the $(p-1)\times(p-1)$ identical matrix.
	Introducing a non-singular matrix
	\begin{align*}
		Y=\left(
		\begin{array}{cc}
			1&\;\bm 0^T\\
			\bm 0&X
		\end{array}	\right),
	\end{align*}
	Properties (a-c) of the Kronecker matrix product  implies that
	\begin{align*}
		(Y\otimes Y)\hat{D}(Y\otimes Y)^T=\left(
		\begin{array}{cc}
			0 & \;\bm 0^T \\
			\bm 0 & \quad\mathbf{I}_{p-1}
		\end{array}	\right)
		\otimes YM_1Y^T+
		YM_1Y^T\otimes\left(
		\begin{array}{cc}
			0 & \;\bm 0^T  \\
			\bm 0 & \quad\mathbf{I}_{p-1}
		\end{array}	\right).
	\end{align*}
	Noting that $YM_1Y^T$ is also SPD,  there exists an orthogonal transformation $W$, such that
	\begin{align*}
		W(YM_1Y^T)W^T=\left(
		\begin{array}{cccc}
			m_1 &  \cdots & 0\\
			\vdots &\ddots&\vdots\\
			0 &  \cdots & m_p
		\end{array}	\right):=M_0,
	\end{align*}
	where the eigenvalues $\{m_i\}_{1\leq i\leq p}$ are all positive real numbers.
	\begin{align*}
		\check{D}=&(W\otimes W)(Y\otimes Y)\hat{D}(Y\otimes Y)^T(W\otimes W)^T\\
		=&\left(
		\begin{array}{cc}
			0 & \quad\bm 0^T \\
			\bm 0 & \quad\mathbf{I}_{p-1}
		\end{array}	\right)
		\otimes M_0
		+ M_0
		\otimes\left(
		\begin{array}{cc}
			0 & \quad\bm 0^T \\
			\bm 0 & \quad\mathbf{I}_{p-1}
		\end{array}	\right).
	\end{align*}
	$\check{D}$ is a diagonal matrix with the following entries on its diagonal line: 
	$$0, m_2, \cdots, m_p, m_1, m_1+m_2,  \cdots, m_1+m_p, \cdots, m_p, m_p+m_2, \cdots, m_p+m_p.$$
	Thus $\check{D}$ has only one zero eigenvalue and $(p^2-1)$ positive 
	ones.
	
	Noting that the first row and the first column of $\hat{D}$ are all zeros, 
	all the eigenvalues of the sub-matrix obtained by removing the first row and the first column of $\hat{D}$  are positive, which leads to $\hat{D}^*_{1,1}\neq 0$.
\end{proof}

%%%%%%%%%%%%%%%%%%%%%%% Documents End %%%%%%%%%%%%%%%%%%%%%%%%%%

%\bibliographystyle{abbrv} % plain, abbrv, siam, ...
%\bibliography{references}

\end{document}